\documentclass[reqno]{amsart}
\usepackage{etex}
\usepackage[T1]{fontenc}
\usepackage[a4paper,hmarginratio=1:1]{geometry}
\usepackage{amssymb,amsfonts,amsmath}
\usepackage{comment} 
\usepackage[all,arc]{xy}
\usepackage{enumerate}
\usepackage{mathrsfs,mathtools}
\usepackage{todonotes,booktabs}
\usepackage{stmaryrd}
\usepackage{marvosym}
\usepackage{graphicx}
\usepackage{pdflscape}  

\usepackage{bbm}
\usepackage{tikz}
\usepackage{tikz-cd}
\usetikzlibrary{trees}
\usetikzlibrary[shapes]
\usetikzlibrary[arrows]
\usetikzlibrary{patterns}
\usetikzlibrary{fadings}
\usetikzlibrary{backgrounds}
\usetikzlibrary{decorations.pathreplacing}
\usetikzlibrary{decorations.pathmorphing}
\usetikzlibrary{positioning}
\newcommand{\cotimes}{\widehat{\otimes}}
\usepackage{xcolor} %improved colors
\usepackage{graphicx}
\SelectTips{cm}{10}

\usepackage[pagebackref]{hyperref}

\definecolor{dark-red}{rgb}{0.5,0.15,0.15}
\definecolor{dark-blue}{rgb}{0.15,0.15,0.6}
\definecolor{dark-green}{rgb}{0.15,0.6,0.15}
\hypersetup{
    colorlinks, linkcolor=dark-red,
    citecolor=dark-blue, urlcolor=dark-green
}

\renewcommand*{\backref}[1]{}
\renewcommand*{\backrefalt}[4]{%
  \ifcase #1 %
No citations.% use \relax if you do not want the "No citations" message
  \or
(cit. on p. #2).%
  \else
(cit on pp. #2).%
  \fi%
}
\usepackage[nameinlink,capitalise,noabbrev]{cleveref}

\newtheorem{theorem}{Theorem}[section]
\newtheorem{hope}[theorem]{Hope}

\newtheorem{corollary}[theorem]{Corollary}

\newtheorem{proposition}[theorem]{Proposition}

\newtheorem{lemma}[theorem]{Lemma}

\newtheorem{conjecture}[theorem]{Conjecture}

\newtheorem*{conj*}{Hovey--Strickland Conjecture}
\newtheorem*{hope*}{Chai's Hope}
\newtheorem*{thm*}{Theorem}
\newtheorem*{thma}{Theorem A}
\newtheorem*{thmb}{Theorem B}
\newtheorem*{thmc}{Theorem C}
\newtheorem*{thmd}{Theorem D}
\newtheorem*{thme}{Theorem E}

\theoremstyle{definition}

\newtheorem{definition}[theorem]{Definition}

\theoremstyle{remark}

\newtheorem{remark}[theorem]{Remark}
 
\makeatletter
\let\c@equation\c@theorem
\makeatother
\numberwithin{equation}{section}

\DeclareMathOperator{\Calg}{CAlg}
\DeclareMathOperator{\Top}{Top}

\DeclareMathOperator{\Sp}{Sp}
\DeclareMathOperator{\Spc}{Spc}
\DeclareMathOperator{\Hom}{Hom}

\DeclareMathOperator{\End}{End}
\DeclareMathOperator{\colim}{colim}
\DeclareMathOperator{\Ext}{Ext}
\DeclareMathOperator{\Lie}{Lie}
\DeclareMathOperator{\Tor}{Tor}
\DeclareMathOperator{\Spec}{Spec}
\DeclareMathOperator{\SpecTop}{SpecTop}
\DeclareMathOperator{\Mod}{Mod}

\DeclareMathOperator{\Comod}{Comod}

\DeclareMathOperator{\Spf}{Spf}

\DeclareMathOperator{\unit}{\mathbf{1}}

\DeclareMathOperator{\Aut}{Aut}

\DeclareMathOperator{\Fil}{Fil}
\DeclareMathOperator{\gr}{gr}

\DeclareMathOperator{\holim}{holim}
\DeclareMathOperator{\supp}{supp}

\newcommand{\cal}{\mathcal}

\newcommand{\Z}{\mathbb{Z}}
\Crefname{figure}{Figure}{Figures}
\Crefname{assu}{Assumption}{Assumptions}
\Crefname{lem}{Lemma}{Lemmas}
\Crefname{conj}{Conjecture}{Conjectures}
\Crefname{thm}{Theorem}{Theorems}
\Crefname{thma}{Theorem}{Theorems}
\Crefname{prop}{Proposition}{Propositions}

\newcommand{\noloc}{\;\mathord{:}\,}
\let\lim\relax

\DeclareMathOperator{\Tot}{Tot}

\newcommand{\fm}{\mathfrak{m}}
\newcommand{\fp}{\mathfrak{p}}

\DeclareMathOperator{\dual}{dual}

\DeclareMathOperator{\lim}{lim}

\newcommand{\htimes}{\widehat \otimes}

\DeclareMathOperator{\cC}{\mathcal{C}}
\DeclareMathOperator{\cD}{\mathcal{D}}

\newcommand{\cK}{\mathcal{K}}

\DeclareMathOperator{\cT}{\mathcal{T}}
\newcommand{\cU}{\mathcal{U}}

\newcommand{\F}{\mathbb{F}}

\DeclareMathOperator{\Gal}{Gal}

\DeclareMathOperator{\CAlg}{CAlg}

\DeclareMathOperator{\filt}{filt}
\DeclareMathOperator{\Tel}{Tel}
\DeclareMathOperator{\Ass}{Ass}

\newcommand{\thick}[1]{\langle #1\rangle}

\title{On conjectures of Hovey--Strickland and Chai}

\author{Tobias Barthel}
\address{Max Planck Institute for Mathematics, Vivatsgasse 7, 53111 Bonn, Germany}
\email{tbarthel@mpim-bonn.mpg.de}

\author{Drew Heard}
\address{Department of Mathematical Sciences, Norwegian University of Science and Technology, Trondheim}
\email{drew.k.heard@ntnu.no}

\author{Niko Naumann}
\address{Fakult{\"a}t f{\"u}r Mathematik, Universit{\"a}t Regensburg, Universit{\"a}tsstraße 31, 
93053 Regensburg, Germany}
\email{Niko.Naumann@mathematik.uni-regensburg.de}

\date{\today}

\begin{document}

\begin{abstract}
We prove the height two case of a conjecture of Hovey and Strickland that provides a $K(n)$-local analogue of the Hopkins--Smith thick subcategory theorem. Our approach first reduces the general conjecture to a problem in arithmetic geometry posed by Chai. We then use the Gross--Hopkins period map to verify Chai's Hope at height two and all primes.  Along the way, we show that the graded commutative ring of completed cooperations for Morava $E$-theory is coherent, and that every finitely generated Morava module can be realized by a $K(n)$-local spectrum as long as $2p-2>n^2+n$. Finally, we deduce consequences of our results for descent of Balmer spectra.
\vspace{-1em}
\end{abstract}

\maketitle

%{\hypersetup{linkcolor=black}\tableofcontents}

\setcounter{tocdepth}{1}
\tableofcontents

\section{Introduction}

The thick subcategory theorem of Hopkins and Smith~\cite{HopkinsSmith1998Nilpotence} is arguably the most important organizational result in the stable homotopy theory of finite complexes. It classifies finite spectra up to the natural operations available on the stable homotopy category: fiber sequences, (de)suspensions, and retracts. The lattice of these equivalence classes of finite spectra resembles the height filtration on the moduli stack of one-dimensional commutative formal groups---a connection between homotopy theory and arithmetic geometry that began in the work of Quillen on complex cobordism. The present paper is concerned with a local analogue of the thick subcategory theorem, which we will relate to the $p$-adic geometry of Lubin--Tate space and in particular to work of Chai. We then exploit this relation by making use of the Gross--Hopkins  period map to establish the height two case of a conjecture of Hovey and Strickland.  

In more detail, let $K(n)$ be the Morava $K$-theory spectrum at height $n\ge 0$ and (implicit) prime $p$. The work of Hopkins and Smith exhibits the categories of $K(n)$-local spectra $\Sp_{K(n)}$ as the constituent pieces of the stable homotopy category. More precisely, $\Sp_{K(n)}$ has two distinguished properties: it is minimal in the sense that it does not admit any further non-trivial localizations, and any finite spectrum can be assembled from its $K(n)$-local pieces along the chromatic tower. From a more abstract point of view, the categories $\Sp_{K(n)}$ for $n>0$ are prominent and naturally occurring examples of compactly generated but not rigidly compactly generated tensor-triangulated categories, i.e., its unit object $L_{K(n)}S^0$ is not compact. 

The structure of the full subcategory of compact objects in $\Sp_{K(n)}$ is rather simple: it is generated by the $K(n)$-localization of any finite type $n$ spectrum $F(n)$ and does not admit any non-trivial thick subcategories. This puts into focus the larger category $\Sp_{K(n)}^{\dual}$ of all $K(n)$-locally dualizable spectra, which in particular contains the much studied $K(n)$-local Picard group \cite{hms_pic}. The goal of this paper is to analyze the global structure of $\Sp_{K(n)}^{\dual}$ and in particular approach the question of how much of the chromatic picture is captured by it. In this direction, Hovey and Strickland \cite[Page 61]{hoveystrickland_memoir} have formulated the following conjecture, which is a local analogue of the thick subcategory theorem~\cite{HopkinsSmith1998Nilpotence}:
\begin{conj*}
If $\cC$ is a thick tensor-ideal of $\Sp_{K(n)}^{\dual}$, then there exists an integer $0 \le k \le n$ and a finite type $k$ spectrum $F(k)$ such that $\cC$ is generated by $L_{K(n)}F(k)$. 
\end{conj*}

The Hovey--Strickland Conjecture is known to be true vacuously for $n=0$. For $n=1$, it can be deduced from the Iwasawa-theoretic classification of weakly dualizable objects in $\Sp_{K(1)}$ by Hahn and Mitchell \cite{HahnMitchell2007Iwasawa} or by direct verification. Our main result is:

\begin{thma}[\Cref{thm:dualizable_n=2}]
The Hovey--Strickland conjecture holds at height $n=2$ and for all primes. 
\end{thma}

This result fits into ongoing efforts to completely understand $K(2)$-local homotopy theory. As we will explain below, progress in heights $n>2$ would require substantially new ideas.

\subsection*{Morava modules and connections to arithmetic geometry}

Our approach to the Hovey--Strickland conjecture is via the arithmetic geometry of Lubin--Tate space, a connection that has proven to be fundamental in the development of local chromatic homotopy theory as envisioned by Morava \cite{morava_noetherian}. Fix an implicit prime $p$ and let $F_n$ be a height $n$ formal group law over $\F_{p^n}$. Lubin and Tate \cite{lubintate} showed that there exists a universal deformation of $F_n$ defined over the commutative ring $E_0 = W\F_{p^n}\llbracket u_1,\ldots,u_{n-1}\rrbracket$; by naturality, the automorphism group $\mathbb{G}_n = \Aut(F_n) \rtimes \Gal(\F_{p^n}/\F_p)$ acts canonically on $E_0$. The Goerss--Hopkins--Miller theorem \cite{goersshopkins} lifts $E_0$ together with its $\mathbb{G}_n$-action to an $\mathbb{E}_{\infty}$-ring spectrum $E=E_n$ called Morava $E$-theory (at height $n$), whose coefficients are given by $E_* = W\F_{p^n}\llbracket u_1,\ldots,u_{n-1}\rrbracket[u^{\pm1}]$ with $W$ being the ring of Witt vectors and $u$ of degree $-2$. A different approach to the construction of Morava $E$-theory using spectral algebraic geometry was recently presented by Lurie~\cite{lurie_ell2}. Devinatz and Hopkins \cite{devinatzhopkins_fixedpoints} exhibit the unit map $L_{K(n)}S^0 \to E_n$ as a homotopical pro-Galois extension with Galois group $\mathbb{G}_n$; informally speaking, the category of $E_n$-modules thus covers $\Sp_{K(n)}$, with descent data encoded by a compatible $\mathbb{G}_n$-action. For a more detailed survey, see for example \cite{bb_chromaticchapter}. 

It follows that the properties of a $K(n)$-local spectrum $X$ are to a large extent controlled by an algebraic invariant called its Morava module, namely $E_*^{\vee}(X) = \pi_*L_{K(n)}(E\otimes X)$. In particular, $X$ is dualizable if and only if its Morava module is finitely generated over $E_*$. Moreover, any Morava module is equipped with an action of the Morava stabilizer group $\mathbb{G}_n$, making it a (twisted) $E_*$-$\mathbb{G}_n$-module. In geometric terms, a Morava module is a $\mathbb{G}_n$-equivariant sheaf on Lubin--Tate space $\Spf(E_0)$. This motivates to study the properties of $K(n)$-locally dualizable spectra via the geometry of Lubin--Tate space together with its $\mathbb{G}_n$-action. 

In \cite{Chai1996group}, Chai studies the Zariski closures of the orbits of the action of the Morava stabilizer group $\mathbb{G}_n$ on the closed fiber of Lubin--Tate space $\Spf(E_0)$. At the end of the paper, he formulates the following: 

\begin{hope*}
Any reduced irreducible formal subscheme of $\Spf(E_0/p)$ stable under an open subgroup of $\mathbb{G}_n$ is empty or the locus $\Spf(E_0/(p,u_1,\ldots,u_{k-1}))$ over which the universal deformation has height at least $k$,  for some $1\le k \le n$. 
\end{hope*}

Chai's Hope for the full $\mathbb{G}_n$-action can be restated as providing a classification of invariant ideals in $E_0$ analogous to Landweber's~\cite{landweber} invariant prime ideal theorem for $BP_*$: The only $\mathbb{G}_n$-invariant radical ideals in $E_0$ are the $I_k=(p=u_0,\ldots,u_{k-1})$ for $0 \le k \le n$ and $E_0$. Our next result establishes a firm link between the global structure of $\Sp_{K(n)}^{\dual}$ and the arithmetic properties of Lubin--Tate space:

\begin{thmb}[\Cref{thm:chaiimplieshs} and \cref{thm:hsimplieschai}]
At a given height $n$ and prime $p$, consider the following two statements:
  \begin{enumerate}
    \item The Hovey--Strickland Conjecture holds.
    \item Chai's Hope for the full $\mathbb{G}_n$-action holds. 
  \end{enumerate}
Then Statement (2) implies Statement (1) and the converse is true provided $2p-2>n^2+n$.
\end{thmb}

Theorem B reduces the proof of Theorem A to the problem of verifying Chai's Hope at height two. This we accomplish by applying the Gross--Hopkins period map~\cite{gross_hopkins}, which relates the $\mathbb{G}_n$-action on $\Spf(E_0)$ to a linear action on rigid projective space through an equivariant map
\[ 
\Phi \colon \Spf(E_0)^{\mathrm{rig}} \otimes \mathbb{Q}_ {p^n}\longrightarrow \mathbb{P}^{n-1}_{\mathbb{Q}_ {p^n}}
\]
of rigid analytic varieties. Our proof of Theorem A relies crucially on the structure of prime ideals in the power series ring $W\F_{p^2}[\![X]\!]$ and can thus not be easily extended to heights larger than two. 

The proof of the implication $(1) \implies (2)$ in Theorem B uses a result of independent interest, namely the realizability of any finitely generated Morava module by a $K(n)$-local spectrum. The condition on the prime $p$ that appears here follows a familiar pattern in local chromatic homotopy theory: for primes $p$ such that $p-1$ does not divide $n$, the group $\mathbb{G}_n$ contains no $p$-torsion, which is the key reason that the category $\Sp_{K(n)}$ is essentially algebraic for $2p-2>n^2+n$.

\begin{thmc}[\Cref{thm:realization}]
If $2p-2>n^2+n$, then any finitely generated Morava module can be realized as the completed $E$-homology of a $K(n)$-locally dualizable spectrum. 
\end{thmc}

The special case of this result for invertible Morava modules is the main theorem of \cite{pstragowski_pic}, and our proof is closely modelled on Pstr\k{a}gowski's argument. In fact, motivated by the algebraicity results of \cite{bss_ultra1, pstragowski_franke, bss_ultra1.5}, we suspect that this theorem can be promoted to an equivalence of categories, but we will not pursue this question at this point.

\subsection*{Tensor-triangular context and descent}

Let us place the above results in a more abstract categorical context. The categories $\Sp_{K(n)}$ can be thought of as the completed stalks of the stable homotopy category $\Sp$ over the points of its Balmer spectrum.\footnote{More precisely, this is true up to the telescope conjecture; a priori, $\Sp_{K(n)}$ is a mild further localization of the completed stalks given by the telescopic categories.} Tensor-triangular geometry provides a wealth of structurally similar examples, for instance in algebraic geometry, modular representation theory, or equivariant stable homotopy theory.
%In each case, 
Under suitable finiteness hypothesis on the ambient tensor triangulated category $\cT$ there is an adelic formalism that assembles $\cT$ from such completed stalk categories $\cT_{\fp}^{\wedge}$, parametrized by the points $\fp \in \Spc(\cT^{\omega})$ of the Balmer spectrum of compact objects in $\cT$, see for example \cite{bg_adelic, amgr_adelic}.

This highlights the importance of understanding the $tt$-geometry of the categories $\cT_{\fp}^{\wedge}$ for the analysis of $\cT$ itself. As for $\Sp_{K(n)}$, the categories $\cT_{\fp}^{\wedge}$ are usually compactly generated, but not rigidly compactly generated, and similar considerations as for $\Sp_{K(n)}$ lead to the study of $\Spc((\cT_{\fp}^{\wedge})^{\dual})$. However, these Balmer spectra remain mysterious; one approach to their computation is via categorical descent along a suitable morphism $\cT_{\fp}^{\wedge} \to \cU$ for a simpler $tt$-category $\cU$. 

In the example $\Sp_{K(n)}$, the proof of the smash product theorem due to Hopkins and Ravenel \cite{ravenel_orangebook} implies that the base-change functor $L_{K(n)}(E_n\otimes-)\colon \Sp_{K(n)} \to \Mod_{E_n}(\Sp_{K(n)})$ satisfies descent, where the target denotes the category of $K(n)$-local $E_n$-modules. A precise formulation of categorical descent requires an enhancement: in the language of $\infty$-categories, the coaugmented cosimplicial diagram of symmetric monoidal $\infty$-categories
\[
\xymatrix{
\Sp_{K(n)}^{\dual} \ar[r] & \Mod_{E_n}^{\omega} \ar@<0.5ex>[r] \ar@<-0.5ex>[r] & \Mod_{L_{K(n)}E_n^{\otimes 2}}^{\omega} \ar@<1ex>[r] \ar@<0ex>[r] \ar@<-1ex>[r] \ar[l] & \Mod_{L_{K(n)}E_n^{\otimes 3}}^{\omega} \ldots \ar@<0.5ex>[l] \ar@<-0.5ex>[l]
}
\]
is a limit diagram.\footnote{It is worthwhile noting that $E_n$ is not $K(n)$-locally dualizable for $n>0$.} This motivates the question whether the Balmer spectrum of $\Sp_{K(n)}^{\dual}$ can be computed in two steps: First determine the Balmer spectra of the categories $\Mod_{L_{K(n)}E_n^{\otimes k+1}}^{\omega}$ for all $k\ge 0$, and then run some kind of descent for Balmer spectra. There is good evidence for the feasibility of such an approach: Under stronger finiteness conditions than what holds in our situation, Balmer \cite{Balmer2016Separable} has indeed established a form of descent for the Balmer spectrum.

The first step in this program is partly realized by the following theorem:
\begin{thmd}
The graded commutative ring $\pi_0L_{K(n)}E^{\otimes k+1} \cong \Hom_{\mathrm{cts}}(\mathbb{G}_n^{\times k}, E_0)$ of completed cooperations for Morava $E$-theory is coherent for all $k\ge 0$. Consequently, Balmer's comparison map 
\[
\xymatrix{
\Spc(\Mod_{L_{K(n)}E^{\otimes k+1}}^{\omega}) \ar[r] & \Spec(\Hom_{\mathrm{cts}}(\mathbb{G}_n^{\times k}, E_0))
}
\]
is surjective for all $k \ge 0$ and a homeomorphism for $k=0$. 
\end{thmd}

This result combines \cref{thm:coherenthighercoop}, \cref{prop:ektheorycomparison}, and \cref{prop:etheorycomparison}, while a proof of the isomorphism $\pi_0L_{K(n)}E^{\otimes k+1} \cong \Hom_{\mathrm{cts}}(\mathbb{G}_n^{\times k}, E_0)$ can be found in \cite[Theorem 12]{strickland_grosshopkins} or \cite[Theorem 4.11]{hovey_operations}. The statement about completed cooperations is dual to the theorem of Hovey and Strickland \cite[Section 2]{hoveystrickland_memoir} that the ring of operations $E^*E$ of Morava $E$-theory is left Noetherian. It is known that $\pi_0L_{K(n)}(E\otimes E)$ is not Noetherian and our proof does not rely on the one by Hovey and Strickland. Instead, we establish and apply a criterion \cref{lem:coherencecriterion2} for a complete ring to be coherent in terms of its associated graded, motivated by a similar criterion for Noetherian rings; our generalization might be well-known to experts, but we could not find it in the literature, so it might be of independent interest.

Returning to the question of descent for Balmer spectra, we observe that our detailed study provides the first example of descent for Balmer spectra which crucially uses that these are {\em spectral} spaces.

\begin{thme}[\Cref{prop:dualizabledescent} and \cref{prop:nodescent}]
For any height $n$ and any prime $p$, the functor $L_{K(n)}(E\otimes-)\colon \Sp_{K(n)}^{\dual} \to \Mod_{E}^{\omega}$ satisfies descent, i.e., it induces an equivalence $\Sp_{K(n)}^{\dual} \simeq \lim_{\bullet\in\Delta}\Mod_{L_{K(n)}E^{\otimes \bullet +1}}^{\omega}$ of symmetric monoidal stable $\infty$-categories. At height $n=2$, the induced diagram 
\[ 
\xymatrix{\Spc(\Mod_{L_{K(2)}(E\otimes E)}^{\omega}) \ar@<0.5ex>[r] \ar@<-0.5ex>[r] & \Spc(\Mod_E^{\omega}) \ar[r] & \Spc(\Sp_{K(2)}^{\dual})}
\]
is not a coequalizer of {\em topological} spaces, but identifies $\Spc(\Sp_{K(2)}^{\dual})$
with the coequalizer in {\em spectral} spaces.
\end{thme}

Its proof relies on Theorem A and Theorem D.
%It would be interesting to reconcile this example with the general descent philosophy in $tt$-geometry.

% \subsection*{Conventions}

% In order to make descent statements such as \eqref{eq:classical}, and more generally for ring spectra, one needs to work with a model such as $\infty$-categories, rather than just homotopy categories. We will sometimes then abuse notation - given a presentable symmetric monoidal stable $\infty$-category $\cC$, and some small essentially small subcategory $\cC^{\#}$, we continue to write $\Spc(\cC^{\#})$ for the Balmer spectrum of $\cC^{\#}$, rather than its homotopy category. Not that there is no ambiguity here, as the Balmer spectrum only depends on the homotopy category of $\cC^{\#}$ in any case.

\subsection*{Acknowledgements}

We would like to thank Andy Baker, Paul Balmer, Ching-Li Chai, Paul Goerss, Piotr Pstr\k{a}gowski, and Neil Strickland for helpful conversations, and the referee for useful comments on an earlier version. We gratefully acknowledge support by the SFB 1085 at the University of Regensburg. The first author would also like to thank the Max Planck Institute for Mathematics for its hospitality. The second author was supported in part by a grant from the Trond Mohn Foundation.

\section{Recollections on tensor triangular geometry}\label{sec:ttgeometry}

We recall some of the basics of the tensor-triangular geometry of an essentially small tensor triangulated category $\cal{K} = (\cal{K},\otimes)$.

\begin{definition}
   A {\em thick subcategory} $\cal{P}$ of $\cal{K}$ is a full triangulated subcategory that is closed under the formation of direct summands. A thick subcategory $\cal{P}\subseteq \cal{K}$ is a {\em thick tensor-ideal} if $\cal{K} \otimes \cal{P} \subseteq \cal{P}$. A thick tensor-ideal is called {\em radical} if $a^{\otimes n} \in \cal{P}$ implies $a \in \cal{P}$, and is called {\em prime} if it is a proper subcategory of $\cal{K}$ and if $a \otimes b \in \cal{P}$ implies $a \in \cal{P}$ or $b \in \cal{P}$. 
\end{definition} 
 
Given a good notion of prime ideal, we can now define the Balmer spectrum of $\cal{K}$ and the support of objects in $\cal{K}$, following \cite{Balmer2005spectrum}. 
 
\begin{definition}
   The {\em Balmer spectrum of} $\cal{K}$, denoted $\Spc(\cal{K})$, is the set of all prime tensor-ideals of $\cal{K}$:
   \[
\Spc(\cal{K}) = \{ \cal{P} \subseteq \cal{K} \mid \cal{P} \text{ is a prime tensor-ideal} \}. 
   \]
   If $\cC$ is a small symmetric monoidal stable $\infty$-category, we also write $\Spc(\cC)$ for the Balmer spectrum of the homotopy category of $\cC$. For any object $X \in \cal{K}$, {\em the support of} $X$ is defined as
   \[
\supp(X) = \{ \cal{P} \in \Spc(\cal{K}) \mid X \not \in \cal{P} \}. 
   \]
\end{definition}
 
There is a topology on $\Spc(\cal{K})$ with $\{\supp(X) \}_{X \in \cal{K}}$ forming a basis of closed subsets. 
The Balmer spectrum is contravariantly functorial with respect to exact tensor triangulated functors $F \colon \cK \to \cal{L}$. Indeed, there is a natural continuous map $\Spc(F) \colon \Spc(\cal{L}) \to \Spc(\cK)$ defined by $\cal{P} \mapsto F^{-1}(\cal{P})$. It has the property that $\supp(F(X)) = \Spc(F)^{-1}(\supp(X))$, see \cite[Proposition 3.6]{Balmer2005spectrum}. 

The interest in the Balmer spectrum comes from its relation to the classification of thick tensor-ideals in $\cal{K}$. For the following, we recall that a Thomason subset of $\Spc(\cal{K})$ is a subset that is a union of closed subsets, each with quasi-compact open complement. The next result is
\cite[Theorem 4.1]{Balmer2005spectrum}.

\begin{theorem}[Balmer]\label{thm:tt_classification}
 Let $\cal{K}$ be an essentially small tensor triangulated category. The assignment
\[
\cal{Y} \mapsto \cal{K}_{\cal{Y}} = \{ X \in \cal{K} \mid \supp(X) \subseteq \cal{Y} \} 
\]
induces an inclusion-preserving bijection between Thomason subsets $\cal{Y}$ of $\Spc(\cal{K})$ and radical thick tensor-ideals of $\cal{K}$. The inverse is given by $\cal{I} \mapsto \supp(\cal{I}) = \cup_{X \in \cal{I}} \supp(X)$.
\end{theorem}

Suppose the symmetric monoidal structure on $\cal{K}$ is closed with internal mapping object denoted by $\underline{\Hom}$. Then $\cal{K}$ is said to be \emph{rigid} if each object is strongly dualizable, i.e., each $X \in \cal{K}$ has a dual $DX = \underline{\Hom}(X,\unit)$ and the canonical map $DX \otimes Y \to \underline{\Hom}(X,Y)$ is an equivalence for every $Y\in\cal{K}$. In this case, every thick tensor-ideal is automatically radical \cite[Proposition 2.4]{Balmer2007Supports}, and so the previous theorem gives a complete description of the thick tensor-ideals of $\cal{K}$ in terms of $\Spc(\cal{K})$.

Let $\langle X \rangle$ denote the thick tensor-ideal generated by a given $X\in\cal{K}$. \Cref{thm:tt_classification} has the following consequence. 

\begin{corollary}\label{lem:support}
  Suppose that $\cK$ is rigid. For $X,Y \in \cK$ we have $X \in \langle Y \rangle$ if and only if $\supp(X) \subseteq \supp(Y)$. 
\end{corollary}

We recall the following definition. 

\begin{definition}
  Suppose $F \colon \cK \to \cal{L}$ is an exact tensor triangulated functor between essentially small tensor triangulated categories. We say that $F$ {\em detects tensor-nilpotence} %of objects
 if every morphism $f \colon X \to Y$ in $\cK$ such that $F(f) = 0$ satisfies $f^{\otimes n} = 0$ for some $n \ge 1$. 
\end{definition}

The next result is shown in \cite[Theorem 1.3]{Balmer2018surjectivity}. 

\begin{theorem}[Balmer]\label{thm:tensor_nilpotence_surjectivity}
  Suppose  $F \colon \cK \to \cal{L}$ detects tensor-nilpotence and that $\cK$ is rigid, then $\Spc(F) \colon \Spc(\cal{L}) \to \Spc(\cK)$ is surjective. 
\end{theorem}

We deduce the following. 

\begin{corollary}\label{lem:thick_equiv}
Suppose  $F \colon \cK \to \cal{L}$ detects tensor-nilpotence and that $\cK$ is rigid, then the following conditions are equivalent for $X,Y \in \cK$: 
\begin{enumerate}
  \item $X \in \langle Y \rangle$. 
  \item $F(X) \in \langle F(Y) \rangle$.
\end{enumerate}
\end{corollary}
\begin{proof}
The implication $\mathit{(1)}\Rightarrow \mathit{(2)}$ is clear and in fact valid for every exact tensor functor $F$.
Conversely, assuming $\mathit{(2)}$, we deduce
\begin{equation}\label{eq:thick_equiv}
\supp(X)=\Spc(F)(\supp(F(X)))\subseteq \Spc(F)(\supp(F(Y)))=\supp(Y),
\end{equation}
where the inclusion follows from our assumption $\mathit{(2)}$, while the first and last equality use the surjectivity of $\Spc(F)$, guaranteed by
\Cref{thm:tensor_nilpotence_surjectivity}. Finally, the containment \eqref{eq:thick_equiv} is equivalent to $\mathit{(1)}$ by \Cref{lem:support} above.
\end{proof}

There are natural comparison maps between $\Spc(\cK)$ and the Zariski spectrum of the endomorphism ring $R^0_{\cK} = \End_{\cK}(\unit)$ as well as the homogeneous spectrum of the graded ring $R_{\cK}^* = \Hom^*_{\cK}(\unit,\unit)$. Under reasonable conditions, these maps are additionally surjective \cite[Theorem 7.3 and Corollary 7.4]{Balmer2010Spectra}.

\begin{theorem}[Balmer]\label{thm:balmer_comparison}
  There exist two continuous maps, natural in $\cK$, 
  \[
\rho_{\cK}^{\ast} \colon \Spc(\cK) \to \Spec^h(R_{\cK}^*) \quad \text{ and } \quad \rho_{\cK} \colon \Spc(\cK) \to \Spec(R_{\cK}^0).
  \]
  Moreover, if $R_{\cK}^*$ is (graded) coherent, then $\rho_{\cK}^*$ and $\rho_{\cK}$ are surjective. 
\end{theorem}

Under much stronger conditions it is known that $\rho_{\cK}^* \colon \Spc(\cK) \to \Spec^h(R_{\cal{K}}^*)$ is a 
homeomorphism \cite[Theorem 1.1 and Lemma 3.10]{DellAmbrogioStanley2016Affine}. 

\begin{theorem}[Dell'Ambrogio--Stanley]\label{thm:ds}
  Suppose that $\cK$ is generated by the tensor unit and that $R_{\cK}^*$ is a graded Noetherian ring concentrated in even degrees such that every homogeneous prime ideal $\mathfrak p\subseteq R_{\cK}^* $ is generated by a (finite) regular sequence of homogeneous elements. Then the comparison map $\rho_{\cK}^*$ is a homeomorphism. Moreover, there is an equality for every $X\in\cK$
  \[
\rho_{\cK}^* (\supp(X)) = \supp_{R_{\cK}^*}(\Hom_{\cK}^*(\unit,X)),
  \]
  where the latter denotes the usual ring theoretic support of a graded $R_{\cK}^*$-module. 
\end{theorem}

Finally, a tensor triangulated category $\cK$ is called {\em local} if $\Spc(\cK)$ is a local topological space, i.e., every open cover $\Spc(\cK) = \bigcup_{i \in I}U_i$ is trivial, in that there exists $i \in I$ such that $U_i = \Spc(\cK)$. We then have the following \cite[Proposition 4.2]{Balmer2010Spectra}.

\begin{proposition}[Balmer]\label{prop:locality}
  Suppose that $\cK$ is rigid, then the following are equivalent:
  \begin{enumerate}
    \item The tensor triangulated category $\cK$ is local. 
    \item The zero ideal is prime (and thus the unique minimal prime). 
    \item If $X \otimes Y = 0$, then $X = 0$ or $Y = 0$. 
  \end{enumerate}
\end{proposition}

\section{Dualizable \texorpdfstring{$K(n)$}{K(n)}-local spectra and realizability of Morava modules}\label{sec:dualizable}

\subsection{Background on \texorpdfstring{$K(n)$}{K(n)}-local homotopy theory}

We begin by recalling some basic notions of $K(n)$-local homotopy theory for some fixed height $n\ge 1$ and an implicit prime $p$. 
Much more can be found in the memoir of Hovey and Strickland \cite{hoveystrickland_memoir}. We recall first that Morava $K$-theory $K(n)$ is an associative ring spectrum with homotopy $\pi_*K(n) \cong \F_{p^n}[u^{\pm 1}]$ where $|u| = 2$. Closely related is the Lubin--Tate spectrum $E_n$, which is a commutative ring spectrum with homotopy 
\[
\pi_*E_n \cong W\F_{p^n}[\![u_1,\ldots,u_{n-1}]\!][u^{\pm 1}]
\] 
with $|u_i| = 0$. Note that this is a complete regular local Noetherian ring (in the graded sense). For clarity, we will fix a height $n$, and often just write $E = E_n$. Recall that $E_n$ itself is $K(n)$-local. Throughout this document, we let $I_k$ denote the ideal $I_k:=(p,u_1,\ldots,u_{k-1}) \subseteq E_0$ for all $0\leq k\leq n$; for example, $I_0 = (0),I_1=(p)$, $I_n\subseteq E_0$ is the maximal ideal, and we set $I_{n+1}=(1)$ by convention. 

We let $\Sp_{K(n)}$ denote the symmetric monoidal $\infty$-category of $K(n)$-local spectra, where the symmetric monoidal structure is given by $K(n)$-localizing the usual smash product of spectra. We will write $\htimes$ for the $K(n)$-localized smash product, reserving the symbol $\otimes$ for the usual smash product of spectra. The homotopy category of $\Sp_{K(n)}$ is a prominent example of a compactly generated tensor triangulated category which, for $n>0$, is not rigidly compactly generated. Its full subcategory of compact objects is the thick subcategory in spectra generated by the $K(n)$-localization of any finite type $n$ spectrum $F(n)$.

From this, one can deduce that $\Spc(\Sp_{K(n)}^{\omega}) = \{(0)\}$,\footnote{Note that the construction of the Balmer spectrum does not require the monoidal structure to be unital.} i.e., there are no non-trivial thick subcategories of compact $K(n)$-local spectra, see \cite[Proposition 12.1]{hoveystrickland_memoir}. In order to capture more of the image of the chromatic filtration in $\Sp_{K(n)}$, it is therefore necessary to consider the larger category $\Sp_{K(n)}^{\dual}$ of $K(n)$-locally dualizable spectra instead. We focus on the latter in this document and observe that its homotopy category is a local rigid tensor triangulated category:
\begin{lemma}\label{lem:minimal_zero}
The tensor triangulated category $\Sp_{K(n)}^{\dual}$ is local. In particular, the minimal prime ideal is the zero ideal. 
\end{lemma}
\begin{proof}
  We verify Condition (3) in \Cref{prop:locality}. Suppose that $X,Y \in \Sp_{K(n)}^{\dual}$ and that $X \htimes Y \simeq 0$. It follows that $K(n)_*(X \htimes Y) \cong K(n)_*X \otimes_{K(n)_*} K(n)_*Y \cong 0$, and because $K(n)_*$ is a graded field, we must have $K(n)_*X = 0$ or $K(n)_*Y = 0$. Because $X$ and $Y$ are $K(n)$-local, this implies that either $X \simeq 0$ or $Y \simeq 0$. 
\end{proof}

The completed $E$-homology of a $K(n)$-local spectrum is defined as $E^{\vee}_*X \coloneqq \pi_*L_{K(n)}(E\otimes X)$. By \cite[Proposition 8.4]{hoveystrickland_memoir} $E^{\vee}_*X$ is always $L$-complete in the sense of \cite[Appendix A]{hoveystrickland_memoir} or, equivalently, derived $I_n$-complete in the sense of \cite[Definition 7.3.0.5]{lurie-sag}. The Morava stabilizer $\mathbb{G}_n$ group acts on $E$, and hence on $E^{\vee}_*X$. This latter action is twisted, and the category of Morava modules is the category with objects $L$-complete $E_*$-modules equipped with a twisted continuous action of $\mathbb{G}_n$. Thus, $E^{\vee}_*X$ is always a Morava module. Working with Morava modules is  equivalent to working with $L$-complete comodules over the $L$-complete Hopf algebroid $(E_*,E^{\vee}_*E)$, see, for example, \cite{BarthelHeard2016term}. We also observe that the category of Morava modules is symmetric monoidal; if $M$ and $N$ are Morava modules, then so is $M \widehat \otimes N \coloneqq L_0(M \otimes_{E_*} N)$, with the diagonal $\mathbb{G}_n$-action, where $L_0$ denotes the $L$-completion functor. If $M$ and $N$ are finitely generated $E_*$-modules, then the $L$-completed tensor product can be replaced by the ordinary $I_n$-adically completed tensor product \cite[Proposition A.4]{hoveystrickland_memoir}.

In \cite[Theorem 8.6]{hoveystrickland_memoir}, Hovey and Strickland characterize dualizability of $K(n)$-local spectra in terms of the associated Morava modules; they prove that the  following conditions are equivalent for $X \in \Sp_{K(n)}$:
  \begin{enumerate}
    \item $X$ is $K(n)$-locally dualizable. 
    \item The Morava module $E_*^\vee X$ of $X$ is finitely generated over $E_*$.
    \item $K(n)_*(X)$ is degreewise finite.
  \end{enumerate}
%In particular, if $X$ and $Y$ are $K(n)$-locally dualizable, then $E^{\vee}_*X \htimes E^{\vee}_*Y$ is isomorphic to the $I_n$-adically completed tensor product of $E^{\vee}_*X$ and $E^{\vee}_*Y$. 

\subsection{A conjecture of Hovey and Strickland}

We now review the construction, due to Hovey and Strickland \cite[Definition 12.14]{hoveystrickland_memoir},  of some thick tensor-ideals in $\Sp_{K(n)}^{\dual}$. 
\begin{definition}\label{defn:ttidealsdk}
  Given $0\leq k \le n$, let $\cal{D}_k$ denote the category of $X \in \Sp_{K(n)}^{\dual}$ such that $X$ is a retract of $Y \otimes Z$ for some $Y \in \Sp_{K(n)}^{\dual}$ and some finite spectrum $Z$ of type at least $k$. We also denote
$\cal{D}_{n+1}:=(0)$ for convenience.
\end{definition}

The objects of $\cal{D}_k$ admit various equivalent characterizations: 

\begin{proposition}[Hovey--Strickland]\label{prop:hs_dk}
The subcategory  $\cal{D}_k\subseteq \Sp_{K(n)}^{\dual}$ is a thick tensor-ideal. 
Moreover, given $X \in \Sp_{K(n)}^{\dual}$ and $0\leq k \leq n+1$, the following are equivalent:
  \begin{enumerate}
    \item $X \in \cal{D}_k$.
    \item The Morava module $E_*^\vee X$ is $I_k$-torsion. 
    \item $X \in \langle L_{K(n)}F(k) \rangle$ for $F(k)$ a finite spectrum of type $k$.
    \item $\Tel(j) \otimes X \simeq 0$ for all $j<k$, where $\Tel(j)$ denotes the telescope of some $v_j$-self map on a finite type $j$ spectrum. 
    \item $K(j)\otimes X \simeq 0$ for all $j<k$.
  \end{enumerate}
In particular, $\cal{D}_k = \langle L_{K(n)}F(k) \rangle$ for any finite type $k$ spectrum $F(k)$.
\end{proposition}
\begin{proof}
The equivalence of $\mathit{(1)}$ and $\mathit{(2)}$ is due to Hovey--Strickland \cite[Proposition 12.15]{hoveystrickland_memoir}. The implication $\mathit{(3)} \implies \mathit{(2)}$ follows from a thick tensor-ideal argument: There is a conditionally and strongly convergent spectral sequence (\cite[Theorem 5.4]{hovey2004ss})
for every $Z,Y\in\Sp_{K(n)}^{\mathrm{dual}}$
\[
E_{s,t}^2\cong\Tor_{s,t}^{E_*}(E_*^{\vee}Z,E_*^{\vee}Y) \implies E_{s+t}^{\vee}(Z \otimes Y).
\]
Since $E_*$ has global dimension $n$, this spectral sequence has a horizontal vanishing line. If now $E_*^{\vee}Z$ is $I_k$-torsion, then so is the $E_2$-page in each bidegree, which shows that $E_{*}^{\vee}(Z \otimes Y)$ is $I_k$-torsion as well. This property is also closed under retracts and long exact sequences, so we see that the Morava module of any $X \in \langle L_{K(n)}F(k) \rangle$ is $I_k$-torsion, since clearly $E_*^{\vee}(L_{K(n)}F(k))$ is $I_k$-torsion to start with. 

Assume Statement $\mathit{(1)}$. To show that $\mathit{(3)}$ holds it suffices to see that $L_{K(n)}Z \in \thick{L_{K(n)}F(k)}$, where $Z$ is a finite spectrum of type at least $k$, which follows from the thick subcategory theorem of Hopkins and Smith~\cite{HopkinsSmith1998Nilpotence}.

The proof of the remaining equivalences was communicated to us by Neil Strickland. By \cite[Proposition 12.15]{hoveystrickland_memoir}, Statement $\mathit{(1)}$ is equivalent to the condition that $X$ is a module over a generalized Moore spectrum $M(k)$ of type $k$, in the sense of \cite[Definition 4.8]{hoveystrickland_memoir}. We will argue that this latter condition in turn is equivalent to Statement $\mathit{(4)}$. Since $M(k) \otimes \Tel(j) \simeq 0$ for all $j<k$, any module over $M(k)$ satisfies the same vanishing condition, so $\mathit{(1)} \implies \mathit{(4)}$. 

In order to prove the converse, we will show that any $X$ satisfying the conditions of $\mathit{(4)}$ admits a module structure over a generalized Moore spectrum $M(j)$ of type $j$ for all $j \le k$. Suppose first that $X$ is a module over $M(j)$ for some generalized Moore spectrum of type $j<k$. Because $X \otimes \Tel(j) \simeq 0$, the $v_j$-self on $M(j)$ must act nilpotently on $X$, so this module structure extends to a module structure over $M(j)/v_{j}^m$ for some $m>0$, see the proof of \cite[Proposition 12.15]{hoveystrickland_memoir}. Note that $M(j)/v_{j}^m$ is a generalized Moore spectrum of type $j+1$. Since $M(0) = S^0$, the claim is vacuously true for $j=0$, whence induction on $j$ shows that any $X \in \cal{D}_k$ is a module over a generalized Moore spectrum $M(k)$ of type $k$. 

Finally, the equivalence of Statements $\mathit{(4)}$ and $\mathit{(5)}$ follows from \cite[Corollary 6.10]{hoveystrickland_memoir}: indeed, this result implies that Bousfield localization at $\Tel(j)$ and at $K(j)$, respectively, coincide on $E_n$-local spectra, hence also on $K(n)$-local spectra. In other words, $\Tel(j) \otimes W \simeq 0$ if and only if $K(j) \otimes W \simeq 0$ for any $K(n)$-local spectrum $W$ and any $j$.
\end{proof}

The Morava module of $X=L_{K(n)}F(k)\in\Sp_{K(n)}^{\dual}$ is $I_k$-torsion but not $I_{k+1}$-torsion, so \Cref{prop:hs_dk} shows
unconditionally that $\cal{D}_{k}\supsetneq\cal{D}_{k+1}$. We thus have a strictly descending chain of tensor-ideals 
\[ 
\Sp_{K(n)}^{\dual}=\cal{D}_0\supsetneq\cal{D}_{1}\supsetneq\ldots\supsetneq\cal{D}_n=\Sp_{K(n)}^{\omega}\supsetneq\cal{D}_{n+1}=(0).
\]
The equality $\cal{D}_n=\Sp_{K(n)}^{\omega}$ here is an immediate consequence of \cite[Theorem 8.5]{hoveystrickland_memoir}. 
Hovey and Strickland conjecture in \cite[Section 12]{hoveystrickland_memoir} that these are all the thick tensor-ideals of $\Sp_{K(n)}^{\dual}$. 

\begin{conjecture}[Hovey--Strickland]\label{conj:hs}
  If $\cC$ is a thick tensor-ideal of $\Sp_{K(n)}^{\dual}$, then $\cC = \cal{D}_k$ for some $0\leq k \le n+1$. 
\end{conjecture}

We next note that this conjecture is equivalent to describing the Balmer spectrum of $\Sp_{K(n)}^{\dual}$.

\begin{proposition}\label{prop:hs_vs_spc}
The following are equivalent:
  \begin{enumerate}
    \item The ideals $\cal{D}_k$ $(0\leq k\leq n+1)$ exhaust all tensor-ideals of $\Sp_{K(n)}^{\dual}$.
    \item We have $\Spc(\Sp_{K(n)}^{\dual})=\{ \cal{D}_1,\ldots,\cal{D}_{n+1}\}$ with topology determined by the
closure operator $\overline{\{ \cal{D}_k \}}=\{ \cal{D}_i\,\mid i\ge k\}$. 
    \item Localization $L_{K(n)}\colon \Sp_{E}^{\omega} \to \Sp_{K(n)}^{\dual}$ induces a homeomorphism of Balmer spectra
        \[
        \xymatrix{\Spc(L_{K(n)})\colon \Spc(\Sp_{K(n)}^{\dual}) \ar[r]^-{\sim} & \Spc(\Sp_{E}^{\omega}).}
        \]
  \end{enumerate}
\end{proposition}

\begin{proof}
To see that $\mathit{(1)}$ implies $\mathit{(2)}$, say $0\leq k\leq n+1$ of the ideals $\cal{D}_1,\ldots,\cal{D}_{n+1}$ are prime.
Given the known inclusions among them, this implies that $\Spc(\Sp_{K(n)}^{\dual})$ has at most $k+1$ closed subsets,
hence $\Sp_{K(n)}^{\dual}$ has at most $k+1$ thick tensor-ideals. This implies $k+1\ge n+2$, hence $k=n+1$, as claimed.

In order to prove the remaining equivalences, we first observe that the localization functor $L_{K(n)}\colon \Sp_E \to \Sp_{K(n)}$ is symmetric monoidal. It therefore induces a symmetric monoidal functor $\phi = L_{K(n)}\colon \Sp_E^{\omega} = \Sp_E^{\dual} \to \Sp_{K(n)}^{\dual}$ by passing to the full subcategories of dualizable objects, and hence a continuous map
\[
\xymatrix{\Spc(\phi)\colon \Spc(\Sp_{K(n)}^{\dual}) \ar[r] & \Spc(\Sp_{E}^{\omega}).}
\]
Let $\cC_k = \langle L_EF(k) \rangle \subseteq \Sp_{E}^{\omega}$ be the thick tensor-ideal of $\Sp_{E}^{\omega}$ generated by the $E$-localization of a finite type $k$ spectrum $F(k)$. The thick subcategory theorem for the $E$-local category~\cite[Theorem 6.9]{hoveystrickland_memoir} shows that the $\cC_k$s are precisely the thick tensor-ideals of $\Sp_{E}^{\omega}$. Together with an argument similar to the one for the implication $(1) \implies (2)$ above, this implies that 
\[
\Spc(\Sp_{E}^{\omega}) = \{\cC_1,\ldots,\cC_{n+1}\} 
\]
with topology determined by the closure operator $\overline{\{ \cC_k \}}=\{ \cC_i\,\mid i\ge k\}$ for all $k \le n+1$. It follows from \cref{prop:hs_dk} that $\phi(\cC_k) \subseteq \cD_k$ and $\phi(\cC_{k-1}) \nsubseteq \cD_{k}$ for all $k$. Since $\phi^{-1}(\cD_k)$ is a thick tensor-ideal of $\Sp_{E}^{\omega}$, we see that $\phi^{-1}(\cD_k) = \cC_k$ for all $k$. 

Assume Statement $\mathit{(2)}$, which says that the thick tensor-ideals $\cD_k$ are prime and that the map $\Spc(\phi)$ is a surjective map between sets of the same finite cardinality, hence a bijection. By inspection of the topologies and using continuity of $\Spc(\phi)$, it must be a homeomorphism, so Statement $\mathit{(3)}$ follows. 

Finally, assume Statement $\mathit{(3)}$, then the number of prime tensor-ideals of $\Sp_{K(n)}^{\dual}$ equals that of $\Sp_{E}^{\omega}$. Therefore, $\Sp_{K(n)}^{\dual}$ has precisely $n+2$ thick tensor-ideals, given the classification of thick tensor-ideals in terms of $\Spc$. This implies that the $\cD_k$ exhaust all tensor ideals of $\Sp_{K(n)}^{\dual}$, which is Statement $\mathit{(1)}$.
\end{proof}

\begin{remark}
In the above situation, we can see unconditionally that the continuous map 
\[
\xymatrix{\Spc(L_{K(n)})\colon \Spc(\Sp_{K(n)}^{\dual}) \ar[r] & \Spc(\Sp_{E}^{\omega})}
\]
is surjective: We compute the pre-image of the stratification by the supports of the ideals $\cC_i$ to be 
\[ 
\Spc(L_{K(n)})^{-1}(\supp(\cC_i=\langle L_EF(i)\rangle)) = \supp(\langle L_{K(n)}F(i)\rangle ) = \supp (\cD_i),
\]
which we know to be a proper stratification. It seems surprising that we cannot unconditionally prove that the ideals $\cD_i$ are indeed prime.
\end{remark}

\subsection{Realizing Morava modules at large primes}

Recall that we are working $p$-locally at a chromatic height $n\ge 1$.
Assume throughout this subsection that $2p-2 >n^2+n$. Our goal is to prove that all finitely generated Morava modules can be realized by $K(n)$-locally dualizable spectra. The special case of invertible comodules is the main result of \cite{pstragowski_pic}. We recall that we write $E = E_n$ and $K = K(n)$ for brevity.

\begin{theorem}\label{thm:realization}
If $2p-2 > n^2 + n$ and $E$ is Morava $E$-theory, any finitely generated Morava module can be realized as the completed $E$-homology of a $K$-locally dualizable spectrum.
\end{theorem}

\begin{remark}
Without the condition on the prime $p$, the conclusion of this result is unlikely to hold in general. For example, if $p=2$ and $n=2$, we suspect that the Morava module $(E_2)_*/(2,u_1)$ cannot be realized as the completed $E$-homology of a $K(2)$-local spectrum. 
\end{remark}

We will need the following purely algebraic finiteness result.

\begin{lemma}\label{lem:extfg}
Let $R = (R,\fm,k)$ be a regular, local, Noetherian commutative ring with maximal ideal $\fm$ and residue field $k = R/\fm$. Then, for every finitely generated $R$-module $N$ and each $i\ge 0$, the $R$-module 
\[
\colim_j\Ext_R^i(N \otimes_{R} R/\fm^j,k) \cong \Ext_R^{i}(N,k)
\]
is finitely generated. Here, the symbol $\otimes_R$ denotes the underived tensor product of $R$-modules.
\end{lemma}
\begin{proof}
The finite generation of the generalized local cohomology $\colim_j\Ext_R^i(N \otimes_{R} R/\fm^j,k)$ also follows from the implication $(c) \implies (a)$ of Theorem 2.9 in \cite{kya_localcohom}; we include an independent argument here. Consider the derived functor on the derived category of $R$-modules
\[
F\colon \cD_R \to \cD_R, \quad C \mapsto \colim_j\mathrm{RHom}_R(C \otimes_{R}^L R/\fm^j,k),
\]
and let $\cC\subseteq\cD_R$ be the full subcategory  of objects $C$ for which $F(C)$ is compact. On the one hand, since 
\[
F(R) \simeq k
\]
by local duality (see \cite[Lemma 2.3]{pstragowski_pic}), we see that $R \in \cC$, hence $\cD_R^{\omega} \subseteq \cC$ because $\cC$ is thick.  On the other hand, $R$ is regular local and thus has finite global dimension, hence every finitely generated $R$-module $N$ is in $\cD_R^{\omega}$. Therefore, we conclude that $F(N) \in \cD_R^{\omega}$, so $H^*F(N)$ is in particular degreewise finitely generated.

There is a composite functor spectral sequence
\[
E_2^{s,t}(N) \cong \colim_j\Ext_R^s(\Tor_t^R(N,R/\fm^j),k) \implies H^*F(N)
\]
which strongly converges because $R$ has finite global dimension. We claim that the $E_2$-term  is concentrated on the $(t=0)$-line, so the spectral sequence collapses. To this end, let $t>0$. We observe that the tower $(\Tor_t^R(N,R/\fm^j))_j$ is nilpotent as a consequence of the Artin--Rees lemma, as follows: Take a partial  free resolution 
\[
0 \to M \xrightarrow{\iota} F=F_{t-1} \to F_{t-2} \to \ldots \to F_0 \to N \to 0
\]
of $N$ with $F_i$ finite free for all $i$. Since $t>0$, dimensional reduction then provides a natural isomorphism
$\Tor_t^R(N,-)\cong\ker(\iota\otimes_R -)$. In particular, given ideals $J\subseteq K\subseteq R$, the canonical map
$\Tor_t^R(N,R/J)\to\Tor_t^R(N,R/K)$ identifies with the
the obvious map $(M\cap JF)/JM\to (M\cap KF)/KM$, and hence is zero if $M\cap JF\subseteq KM$. Now, by the Artin--Rees lemma, there is an $i_0\ge 0$
such that for all $i\ge i_0$ we have 
\[ 
\fm^iF\cap M\subseteq \fm^{i-i_0}(\fm^{i_0}F\cap M)\left( \subseteq \fm^{i-i_0}M\right),
\]
which by the above means that $\Tor_t^R(N,R/\fm^i)\to\Tor_t^R(N,R/\fm^{i-i_0})$ is zero. Hence the tower $(\Tor_t^R(N,R/\fm^j))_j$ is nilpotent, as claimed.
This implies that, for every  $s$, the inductive system
\[
(\Ext_R^s(\Tor_t^R(N,R/\fm^j),k))_j
\]
is nilpotent, too, hence
\[
E_2^{s,t}(N) \cong \colim_j\Ext_R^s(\Tor_t^R(N,R/\fm^j),k) = 0
\]
for all $s \ge 0$ and all $t\ge 1$.  
Consequently, the spectral sequence degenerates into an isomorphism, for every $s\ge 0$:
\[ 
E_2^{s,0}\cong  \colim_j\Ext_R^s(N\otimes_R R/\fm^j,k)\cong
H^sF(N).
\]
Since we already observed that $H^sF(N)$ is finitely generated,
this finishes the proof that the groups $\colim_j\Ext_R^i(N \otimes_{R} R/\fm^j,k)$ are finitely generated.

Finally, note that we have natural isomorphisms 
\[
F(N) = \colim_j\mathrm{RHom}_R(N \otimes_{R}^L R/\fm^j,k) \simeq \mathrm{RHom}_R(N, \colim_j\mathrm{RHom}_R(R/\fm^j,k)) \simeq \mathrm{RHom}_R(N, k),
\]
using that $N$ is perfect and local duality again. By the collapse of the above spectral sequence, we thus obtain an isomorphism $\colim_j\Ext_R^i(N\otimes_R R/\fm^j,k) \cong \Ext_R^{i}(N,k)$, as desired.
\end{proof}

The remainder of the argument proceeds along the same lines as the proof of \cite[Theorem 2.5]{pstragowski_pic}. Let $M$ be a Morava module, finitely generated as a graded $E_*$-module.

\begin{lemma}
For each $j \ge 0$, there is a canonical isomorphism
\[
E^{\vee}_*E \cotimes (M \otimes E_*/I_n^j) \cong E_*E \otimes (M \otimes E_*/I_n^j).
\]
In particular, $(M \otimes E_*/I_n^j)$ is a comodule over the Hopf algebroid $(E_*,E_*E)$.
\end{lemma}
\begin{proof}
This follows because $E_*/I_n^j$ is $I_n$-power torsion (cf.~\cite[Remark 1.4]{BarthelHeard2016term}). 
\end{proof}

The next result follows from the case $k = 1$ of Theorem 2.15 in \cite{pstragowski_franke} (see also Definition 2.16 and Remark 2.17), following earlier work by Bousfield.

\begin{lemma}[Pstr\k{a}gowski]
For $2p-2 > n^2+n$ there exists a functor $\beta\colon \Comod_{E_*E} \to h\Sp_E$ such that $E_*\beta(X) \cong X$ as $E_*E$-comodules for any $X \in \Comod_{E_*E}$.
\end{lemma}

Combining the two previous lemmas, we obtain for all $j\ge 0$ an $E_*E$-comodule isomorphism $E_*\beta(M \otimes E_*/I_n^j) \cong M \otimes E_*/I_n^j$, which in turn is isomorphic to $E_*^{\vee}\beta(M \otimes E_*/I_n^j)$ because it is an $I_n$-power torsion module. 

\begin{lemma}\label{lem:fg}
For each $j\ge 0$, the graded $E_*$-module $\colim_jK^*\beta(M \otimes E_*/I_n^j)$ is finitely generated.
\end{lemma}
\begin{proof}
There is a universal coefficient spectral sequence \cite[Theorem 4.1]{hovey2004ss}
\[
\Ext_{E_*}^*(E_*^{\vee}\beta(M \otimes E_*/I_n^j), K_*) \implies K^*\beta(M \otimes E_*/I_n^j)
\]
which converges with a horizontal vanishing line of intersect $n$ on the $E_2$-page since $E_0$ is regular of dimension $n$. Since filtered colimits are exact, we obtain a convergent spectral sequence
\[
\colim_j\Ext_{E_*}^*(E_*^{\vee}\beta(M \otimes E_*/I_n^j), K_*) \implies \colim_jK^*\beta(M \otimes E_*/I_n^j),
\]
which also has a horizontal vanishing line of intersect $n$ on the $E_2$-page. \Cref{lem:extfg} shows that the terms on the $E_2$-page of this spectral sequence are finitely generated $E_*$-modules, which implies the claim because $E_*$ is Noetherian.  
\end{proof}

\begin{proof}[Proof of \cref{thm:realization}]
Let $M$ be a finitely generated Morava module. We define $X_j(M) = L_K\beta(M \otimes E_*/I_n^j)$, and $X(M)$ as the homotopy limit $\holim_jX_j(M)$, computed by lifting the tower to the underlying model, taking the homotopy limit there, and then passing to $h\Sp_E$. We have already established the following isomorphism
\[
E_*^{\vee}X_j(M) \cong M \otimes E_*/I_n^j.
\]
In particular, as this is a finitely generated $E_*$-module, this shows that $X_j(M) \in \Sp_{K(n)}^{\dual}$ is $K$-locally dualizable for all $j \ge 0$. As a next step, we show that the spectrum $X(M)$ is $K$-locally dualizable.

To this end, note first that $X(M) \simeq D_K\colim_jD_KX_j(M)$, so it suffices to prove that $\colim_jD_KX_j(M)$ is $K$-locally dualizable, where $D_K$ denotes $K$-local Spanier--Whitehead duality and the colimits are computed in the $K$-local category. By \cite[Theorem 8.6]{hoveystrickland_memoir}, this is the case if and only if $K_*\colim_jD_KX_j(M)$ is degreewise finite. We compute:
\[
K_*\colim_jD_KX_j(M) \cong \colim_jK_*D_KX_j(M) \cong \colim_jK^*X_j(M). 
\]
\Cref{lem:fg} shows that this graded $E_*$-module is finitely generated over $E_*$. Moreover, it is $I_n$-torsion, hence finitely generated over $K_*$, and thus degreewise finite, as desired.

It remains to establish an isomorphism $E_*^{\vee}X(M) \cong M$ of Morava modules. Dualizability of $X(M)$ and $X_j(M)$ for all $j\ge 0$ allows us to apply \cite[Lemma 2.4]{pstragowski_pic} to get a $K$-local equivalence $L_K(E \otimes X(M)) \simeq \holim_jL_{K}(E \otimes X_j(M))$. The associated Milnor sequence then provides a short exact sequence of $E_*$-modules
\[
0 \to \lim_j^1(M \otimes E_{*+1}/I_n^j) \to E_*^{\vee}X(M) \to \lim_j (M \otimes E_*/I_n^j) \to 0.
\]
Since all terms in the limit are degreewise finite, the $\lim^1$-term vanishes and, because $M$ is $L$-complete, this sequence then degenerates to an isomorphism $E_*^{\vee}X(M) \cong M$. Because this isomorphism is induced by a map which is compatible with the $E_*E$-comodule structure, it is also an isomorphism of Morava modules, by virtue of the canonical completion map $E_*E \to E^{\vee}_*E$. 
\end{proof}

\section{Relation to arithmetic geometry}\label{sec:chaihope}

We will relate \Cref{conj:hs} to the following hope, stated by Chai \cite[p.~753]{Chai1996group}. A similar statement is formulated in \cite[Problem 16.8]{hoveystrickland_memoir}.

\begin{hope}[Chai]\label{conj:chai}
  The only $\mathbb{G}_n$-invariant radical ideals in $E_0$ are the $I_k$ for $0 \le k \le n$. 
\end{hope}

\begin{remark} 
This statement deviates from Chai's in two respects. Firstly, Chai considers only $E_0/p$ rather than $E_0$ itself, and secondly, he considers more generally ideals invariant under open subgroups of $\mathbb{G}_n$, rather than only $\mathbb{G}_n$. Since for our present considerations only the full action is relevant, we only consider this special case of Chai's Hope. The distinction between $E_0/p$ and $E_0$ should be irrelevant, according to the following result of Chai (unpublished): The only invariant ideals of $E_0\otimes\mathbb{Q}$ are $(0)$ and $(1)$. One should be able to see that by using the differentiability of the action of $\mathbb{G}_n$
on $E_0\otimes\mathbb{Q}$ and the fact that the resulting action of the Lie-algebra can be made completely explicit.
\end{remark}

\subsection{$K(n)$-local $E$-modules}

In order to relate \cref{conj:chai} to \cref{conj:hs}, we begin with a few preliminaries. Let $\Mod_E(\Sp_{K(n)})$ be the $\infty$-category of $K(n)$-local $E$-modules, the category of $E$-modules internal to $\Sp_{K(n)}$. This category inherits a symmetric monoidal structure from $\Sp_{K(n)}$, given by the $K(n)$-local smash product over $E$, and we denote by $\Mod_E(\Sp_{K(n)})^{\dual}$ its full subcategory of dualizable objects. We will write $\Mod_E^{\omega}$ for the $\infty$-category of compact $E$-modules; note that every compact $E$-module is automatically $K(n)$-local. 

The next lemma will be used repeatedly below to translate between homogeneous and ordinary Zariski spectra. 

\begin{lemma}\label{lem:gradedungradedcomparison}
If $R_*$ is a 2-periodic evenly graded commutative ring, then 
\[
\xymatrix{\mathfrak q\mapsto \mathfrak q\cap R_0\colon \Spec^h(R_*)  \ar@<0.5ex>[r] &  \Spec(R_0) \ar@<0.5ex>[l] \noloc \fp\cdot R_*\mapsfrom \fp}
\] 
are inverse homeomorphisms.
\end{lemma}

We will write $\supp_E$ for the $tt$-support function on $\Mod_{E}^{\omega}$, while $\supp_{E_*}$ and $\supp_{E_0}$ denote the Zariski support.

\begin{lemma}\label{prop:etheorycomparison}
The graded comparison map 
\[
\rho^*_{\Mod_{E}^{\omega}}\colon\Spc(\Mod_{E}^{\omega}) \to \Spec^h(E_*) \cong \Spec(E_0)
\]
is a homeomorphism, and for every $X\in\Mod_E^{\omega}$ we have 
    \[ 
    \rho^*_{\Mod_E^{\omega}}(\supp_E(X))=\supp_{E_*}(\pi_*(X)) =\supp_{E_0}(\pi_0(X)\oplus\pi_1(X))
    \]
under the identification of \cref{lem:gradedungradedcomparison}.
\end{lemma}
\begin{proof}
Since $E_*$ is an even regular Noetherian commutative graded ring, this is a consequence of \Cref{thm:ds} and \cref{lem:gradedungradedcomparison}.
\end{proof}

Since $E\in\CAlg(\Sp_{K(n)})$ is a commutative algebra,
the base-change functor
\[ 
E\htimes - \colon \Sp_{K(n)} \to \Mod_{E}(\Sp_{K(n)})
\]
is symmetric monoidal and thus restricts to a functor $\Sp_{K(n)}^{\dual} \to \Mod_E^{\omega}$.

\begin{lemma}\label{cor:invaraint_subset2}
For any $X \in \Sp_{K(n)}^{\dual}$ the support $\supp_E(E \htimes X)$ identifies via
\Cref{prop:etheorycomparison} with a $\mathbb{G}_n$-invariant subset of $\Spec(E_0)$. 
\end{lemma}
\begin{proof}
  Let $g \in \mathbb{G}_n$, then $g$ gives rise to a map of commutative algebras $g \colon E \to E$, which in turn induces a functor $F_g \colon \Mod_E^{\omega} \to \Mod_E^{\omega}$ via $F_g:=-\otimes_{E,g} E$.
Under the identification of \Cref{prop:etheorycomparison} the induced map $\Spc(F_g) \colon \Spc(\Mod_E^{\omega}) \to \Spc(\Mod_E^{\omega})$ is the map $\Spec(\pi_0(g)) \colon \Spec(E_0) \to \Spec(E_0)$, as follows from the 
naturality of the comparison map $\rho_{\Mod_E^\omega}^*$.

  Now suppose that $X \in \Sp_{K(n)}^{\dual}$, then $\supp_E(F_g(E \htimes X)) = \Spc(F_g)^{-1}(\supp_E(E \htimes X))$
identifies with $\Spec(\pi_0(g))^{-1}( \supp_E(E \htimes X))$ by properties of the Balmer spectrum and the previous paragraph. But $\supp_E(F_g(E \htimes X)) = \supp_E(E \htimes X)$ because multiplication by $g^{-1}$
gives an equivalence of $E$-modules $F_g(E\htimes X)\simeq E\htimes X$ (this uses that $E\htimes X$ is an induced $E$-module). Hence $\supp_E(E \htimes X)$ identifies with $\Spec(\pi_0(g))^{-1}( \supp_E(E \htimes X))$. 
As $g\in\mathbb{G}_n$ was arbitrary, this proves the claim.
\end{proof}

\begin{proposition}\label{prop:tensor_nilpotence}
  The functor  $E\htimes - \colon \Sp_{K(n)} \to \Mod_{E}(\Sp_{K(n)})$ detects tensor-nilpotence, and hence so does its restriction to the full subcategories of dualizable objects. 
\end{proposition}
\begin{proof}
By \cite[Proposition 3.26]{mathew_galois}, the first statement of the proposition is equivalent to show that $E \in \CAlg(\Sp_{K(n)})$ admits descent, i.e., the thick tensor-ideal generated by $E$ is all of $\Sp_{K(n)}$. This follows from the proof of the smash product theorem \cite[Chapter 8]{ravenel_orangebook}, see for example \cite[Proposition 10.10]{mathew_galois}. 
\end{proof}

The category of dualizable objects in $\Mod_{E}(\Sp_{K(n)})$ admits a more familiar description, due to Mathew~\cite[Proposition 10.11]{mathew_galois}:

\begin{proposition}[Mathew]\label{prop:dual_perfect}
The inclusion functor $\Mod_E^{\omega} \xrightarrow{\sim} \Mod_E(\Sp_{K(n)})^{\dual}$ from compact $E$-modules to dualizable $K(n)$-local $E$-modules is an equivalence. 
\end{proposition}

Combining \cref{thm:tensor_nilpotence_surjectivity,prop:tensor_nilpotence,prop:etheorycomparison} we deduce the following. 

\begin{corollary}\label{cor:surjective_balmer}
The map 
\[ 
\Spec(E_0)\longrightarrow\Spc(\Sp_{K(n)}^{\dual}),\,\mathfrak p\mapsto P(\mathfrak p):=\left\{ X\in
\Sp_{K(n)}^{\dual}\,\mid\, \left( E_0^\vee(X)\oplus E_1^\vee(X)\right)_{\mathfrak p} = 0\right\}
\]
is well-defined, continuous and surjective. 
\end{corollary}

Observe that using the Tor spectral sequence it is easy to see directly that $P(\mathfrak p)\subseteq \Sp_{K(n)}^{\dual}$
is a thick tensor-ideal. It seems less immediate, however, that it is indeed prime.

\begin{proof}[Proof of \Cref{cor:surjective_balmer}]
We abbreviate $\rho:=\rho^*_{\Mod_E^\omega}$.
Consider the composite 
\[ 
\xymatrixcolsep{4pc}
\xymatrix{
\Spec(E_0)\simeq\Spec^h(E_*) \ar[r]_-{\sim}^-{\rho^{-1}} & \Spc(\Mod_E^\omega)\ar[r]^-{\Spc(E\htimes -)} & \Spc(\Sp_{K(n)}^{\dual}),
}
\]
which is continuous and surjective (using \Cref{thm:tensor_nilpotence_surjectivity} and \Cref{prop:tensor_nilpotence} for the surjectivity). 
Our task is to show that this map is given by $\mathfrak p\mapsto\ P(\mathfrak p)$.
For this, we fix $X\in\Sp_{K(n)}^{\dual}$ and reformulate in several steps the condition that $X\in\Spc(E\htimes-)(\rho^{-1}(\mathfrak p))$.
By the definition of $\Spc(E\htimes -)$, it is equivalent to demand that $E\htimes X\in \rho^{-1}(\mathfrak p)$, i.e., that 
\[ 
\supp_E(E\htimes X)\subseteq\supp_E(\rho^{-1}(\mathfrak p))=\{ \mathfrak q\textbf{}\,\mid\, \rho^{-1}(\mathfrak p)\nsubseteq \mathfrak q\},
\]
by the definition of support of an ideal. If we apply $\rho$ which is bijective and reverses the inclusion of prime ideals
and recall \Cref{prop:etheorycomparison}, we see that this condition is equivalent to 
\[ 
\supp_{E_0}(E_0^\vee(X)\oplus E_1^\vee(X))=\rho(\supp_E(E\htimes X))\subseteq \{ \mathfrak q\,\mid\,  \mathfrak q\nsubseteq \mathfrak p \}.
\]
This means that for all $\mathfrak q\subseteq \mathfrak p$ we have $(E_0^\vee(X)\oplus E_1^\vee(X))_{\mathfrak q} =0$.
Since the $E_0$-module $E_0^\vee(X)\oplus E_1^\vee(X)$ is finitely generated, the Nakayama lemma shows that these conditions for all
$\mathfrak q\subseteq \mathfrak p$ are equivalent to the single condition $(E_0^\vee(X)\oplus E_1^\vee(X))_{\mathfrak p} =0$, i.e., to
$X\in P(\mathfrak p)$.
\end{proof}

\subsection{The relation between the Hovey--Strickland Conjecture and Chai's Hope}

Recall the definition of the thick tensor-ideals $\cal{D}_k \subseteq \Sp_{K(n)}^{\dual}$ from \cref{defn:ttidealsdk}.

\begin{theorem}\label{thm:chaiimplieshs}
  If \Cref{conj:chai} holds, then every thick tensor-ideal of $\Sp_{K(n)}^{\dual}$ is equal to $\cal{D}_k$ for some $0 \leq k \le n+1$. In particular, \Cref{conj:chai} implies \Cref{conj:hs}. 
\end{theorem}
\begin{proof}
  Let $(0)\neq\cC$ be a thick tensor-ideal of $\Sp_{K(n)}^{\dual}$. Because $\cal{D}_0 = \Sp_{K(n)}^{\dual}$, there is a largest $k$ such that $\cC \subseteq \cD_k$ and since $\cC\neq (0)=\cal{D}_{n+1}$, we have $0\leq k\leq n$. By our choice of $k$ and \Cref{prop:hs_dk} there is some $X \in \cC$
such that $E_*^\vee X$ is $I_k$-torsion, but not $I_{k+1}$-torsion. By \Cref{prop:etheorycomparison} $\rho^*_{\Mod_E^\omega}(\supp_E(E \htimes X)) = \supp_{E_*}(E_*^\vee X)$, and hence by \cite[Lemma 2.2 and Lemma 2.4]{BensonIyengarKrause2008Local} we have 
\begin{equation}\label{eq:supp_torsion}
\rho^*_{\Mod_E^\omega}(\supp_E(E \htimes X)) = \supp_{E_*}(E_*^\vee X) \subseteq V(I_k) \mbox{ but }
\end{equation}
\[ 
\rho^*_{\Mod_E^\omega}(\supp_E(E \htimes X)) = \supp_{E_*}(E_*^\vee X) \not\subseteq V(I_{k+1}).
\]

Let $L_{K(n)}F(k)$ be a generator of $\cal{D}_k$, i.e., the $K(n)$-localization of a finite spectrum $F(k)$ of type $k$. We want to show that $\cD_k \subseteq \cC$ which follows if $L_{K(n)}F(k) \in \langle X \rangle$. By \cref{lem:thick_equiv,cor:surjective_balmer} it is enough to show that $E \htimes F(k) \in \langle E \htimes X \rangle$. Observe that using \Cref{prop:etheorycomparison} again gives $\rho^*_{\Mod_E^\omega}(\supp_E(E \htimes F(k))) = \supp_{E_*}(E_*^{\vee}F(k)) = V(I_k)$, and so by \Cref{lem:support} it suffices to establish the inclusion  $V(I_k) \subseteq \rho^*_{\Mod_E^\omega}(\supp(E \htimes X))$. 

By \Cref{cor:invaraint_subset2} the support of $\rho^*_{\Mod_E^\omega}(\supp_E(E \htimes X))$ is a $\mathbb{G}_n$-invariant subset of $\Spec(E_0)$, and hence by \Cref{conj:chai} is of the form $V(I_t)$ for some $0 \le t \le n$. However, by \eqref{eq:supp_torsion} we must therefore have that $\rho^*_{\Mod_E^\omega}(\supp_E(E \htimes X)) = V(I_k)$, and hence $\cC = \cal{D}_k$ by \cref{prop:hs_dk}, so we are done. 
\end{proof}

\begin{corollary}\label{cor:dualizable_n=1}
The Balmer spectrum $\Spc(\Sp_{K(1)}^{\dual}) = \{ \cal{D}_1,\cal{D}_2 \}$ is a two point space with $(0)=\cal{D}_2 \subsetneq \cal{D}_1=\Sp_{K(1)}^\omega$. Consequently, the Hovey--Strickland Conjecture holds for $n=1$ and any prime.
\end{corollary}
\begin{proof}
When $n = 1$ we have $E_0 \cong \Z_p$ with $\mathbb{G}_1$ acting trivially, so it is clear that \Cref{conj:chai} holds and the result follows from \cref{thm:chaiimplieshs} and \cref{prop:hs_vs_spc}.
\end{proof}

\begin{remark} For odd primes, Hahn and Mitchell have computed the set of {\em all} thick subcategories 
(not necessarily tensor-ideal) of $\Sp_{K(1)}^{\dual}$ in Iwasawa-theoretic terms \cite[Section 10]{HahnMitchell2007Iwasawa}. It is not too hard to check directly that the only primes they find are the two given by \Cref{cor:dualizable_n=1}, but we will not include a detailed proof of this. By contrast, finding all thick subcategories of $\Sp_{K(2)}^{\dual}$ seems out of reach.
\end{remark}

For sufficiently large primes, the converse also holds. 

\begin{theorem}\label{thm:hsimplieschai}
  Suppose that $2p-2 > n^2+n$. Then \Cref{conj:hs} implies \Cref{conj:chai}. 
\end{theorem}
\begin{proof}
Let $I \subseteq E_0$ be an invariant radical ideal. By \Cref{thm:realization} we can find $X \in \Sp_{K(n)}^{\dual}$ such that $E_*^\vee X \cong E_*/I$. In particular,
\[
\rho^*_{\Mod_E^\omega}(\supp_E(E \htimes X)) = \supp_{E_*}(E_*^\vee X) = V(I), 
\]
where we use \Cref{prop:etheorycomparison} again. 

Now consider the thick tensor-ideal $\langle X \rangle$. Assuming \Cref{conj:hs} we must have that $\langle X \rangle = \langle L_{K(n)}F(k) \rangle$ for some $0 \le k \le n+1$ (with $F(n+1)=0$). 
By \Cref{lem:thick_equiv} $\langle E \htimes X \rangle = \langle E \htimes F(k) \rangle$. In particular, 
\[ \rho_{\Mod_E^\omega}^*(\supp_E(E \htimes X)) = \rho_{\Mod_E^\omega}^*(\supp_E(E \htimes F(k))) = V(I_k).\] 
Hence $V(I_k) = V(I)$. However, then $I_k=\sqrt{I_k} = \sqrt{I}=I$, and so \Cref{conj:chai} holds. 
\end{proof}
\begin{corollary}
  For $2p-2 > n^2+n$, \Cref{conj:hs,conj:chai} are equivalent.
\end{corollary}

\subsection{The case $n = 2$}
The goal of this section is to verify Chai's Hope when $n = 2$, and hence to prove the following. 
\begin{theorem}\label{thm:dualizable_n=2}
  The Balmer spectrum $\Spc(\Sp_{K(2)}^{\dual}) = \{ \cal{D}_1,\cal{D}_2, \cal{D}_3 \}$ with $(0)=\cal{D}_3 \subsetneq \Sp_{K(2)}^\omega = \cal{D}_2 \subsetneq \langle S^0/p \rangle =\cal{D}_1$. 
\end{theorem}

The key here is the following.

\begin{theorem}\label{thm:invariant_primes}
If the prime $\mathfrak{p}\subseteq E_0\simeq W\mathbb{F}_{p^2}[\![X]\!]$ is associated with a $\mathbb{G}_2$-invariant ideal $I\subseteq W\mathbb{F}_{p^2}[\![X]\!]$,
then $\mathfrak{p}\in\{ I_0=(0), I_1=(p), I_2=(p,X)\}$.
\end{theorem}

Before engaging into the proof, we note this implies Chai's Hope in height two.

\begin{corollary}
Every invariant radical ideal of $E_0$ is one of $I_0,I_1$ and $I_2$ (and in particular, is prime).
\end{corollary}

\begin{proof} Let $I\subseteq E_0$ be an invariant radical ideal and $I=\bigcup_i Q_i$ a primary decomposition. This means that $V(I)=\bigcup_i V(\sqrt{Q_i})$ is the decomposition of $V(I)$ into irreducible components. Each $\sqrt{Q_i}$ is a prime associated with the invariant ideal $I$, hence equal to one of the $I_i$ by \cref{thm:invariant_primes}. Since the $I_i$ are linearly ordered
under inclusion, we have $V(I)=V(I_i)$ for some $i$, hence $I=I_i$, because both ideals are radical.
\end{proof}

We denote by $\mathbb{S}_2\subseteq\mathbb{G}_2$ the non-extended stabilizer group, i.e., the units of the maximal order of the unique division algebra
over $\mathbb{Q}_p$ of invariant $1/2$.
Our first step in the proof of \cref{thm:invariant_primes} is the non-triviality statement \Cref{prop:orbitsize} below, about the canonical action of $\mathbb{S}_2$ on
 $\mathbb{P}^1_{\mathbb{Q}_{p^2}}$ through (Galois-twisted) fractional linear transformations. We recall the formula for this action below in \Cref{eq:explicit_action}. We give two related arguments
for this. The first one freely uses basic results about $p$-adic manifolds, for which we refer to \cite[Part II, Chapter III]{serre_lie}.
This argument generalizes to arbitrary heights. The second argument makes explicit the $p$-adic exponential map involved in this theory, and hence reduces the argument to an 
elementary manipulation of power series.\\
The group $\mathbb{S}_2$ is a compact $p$-adic Lie-group of dimension $4$ 
which admits a faithful representation $\mathbb{S}_2\subseteq \mathrm{Gl}_2(W\mathbb{F}_{p^2})$ through which it acts (by right multiplication) on $\mathbb{P}^1_{\mathbb{Q}_ {p^2}}$.
In particular, $\mathbb{S}_2$ is a $\mathbb{Q}_p$-manifold of dimension $4$. The rigid-analytic space $\mathfrak{X} \coloneqq \Spf(E_0)^{\mathop{rig}}$ is the open unit-disc over $\mathbb{Q}_{p^2}$ and supports the $\mathbb{S}_2$-equivariant period map
\[ 
\Phi \colon \mathfrak{X}\longrightarrow \mathbb{P}^1_{\mathbb{Q}_ {p^2}}.
\]
See \cite{gross_hopkins}, or \cite[Section 1]{kohlhaase} for a succinct summary of this. We denote by $\mathbb{C}_p$ the completion of an algebraic closure of $\mathbb{Q}_p$. This is a complete ultrametric and algebraically closed field. Fix a point $x\in\mathbb{P}^1(\mathbb{C}_p)$. The stabilizer of $x$ is a closed subgroup $U\subseteq\mathbb{S}_2$, hence the orbit  $x\mathbb{S}_2\simeq U\setminus \mathbb{S}_2$ has uniquely the structure of a $\mathbb{Q}_p$-manifold such that the map $\mathbb{S}_2\longrightarrow x\mathbb{S}_2$, $g\mapsto xg$ is a submersion. We now observe that the orbit is not finite, i.e., not of dimension $0$.

\begin{proposition}\label{prop:orbitsize}
In the above situation, we have $dim_x(x\mathbb{S}_2)\ge 1$.
\end{proposition}
\begin{proof}[First proof of \Cref{prop:orbitsize}]
This will be a computation of tangent spaces. We denote by $\mathcal{O}_{\mathbb{C}_p}\subseteq\mathbb{C}_p$ the ring of integers and recall the representation of the projective line
{\em over } $\mathbb{C}_p$ 
 \[ 
 \mathrm{P}(\mathcal{O}_{\mathbb{C}_p})\setminus \mathrm{Gl}_2(\mathcal{O}_{\mathbb{C}_p})\simeq
\mathbb{P}^1(\mathbb{C}_p)\, , \, \mathrm{P}(\mathcal{O}_{\mathbb{C}_p})g\mapsto [1:0]g
\]
as a principal homogeneous space. Here,
\[
P:=\left\{\begin{pmatrix}
* & 0 \\
* & *
\end{pmatrix}\right\} \subseteq \mathrm{Gl}_2
\]
denotes the corresponding parabolic subgroup of $\mathrm{Gl}_2$.\\
Choose some $g\in  \mathrm{Gl}_2(\mathcal{O}_{\mathbb{C}_p})$ such that $x=[1:0]g$. We will compute 
the tangent space at $x\in x\mathbb{G}_2$ as a $\mathbb{Q}_p$-linear subspace  
\[
T_x(x\mathbb{G}_2)\subseteq T_x(x\mathrm{Gl}_2(\mathcal{O}_{\mathbb{C}_p})=\mathbb{P}^1(\mathbb{C}_p))
\]
inside this 1-dimensional $\mathbb{C}_p$-vector space.
Since the stabilizer of $x$ is $g^{-1}\mathrm{P}(\mathcal{O}_{\mathbb{C}_p})g$ we first obtain an exact sequence of $\mathbb{C}_p$-vector spaces
\begin{equation}\label{eq:1}
\xymatrix{ 0 \ar[r] & \Lie(g^{-1}\mathrm{P}(\mathcal{O}_{\mathbb{C}_p})g)  \ar[r] & \Lie(\mathrm{Gl}_2(\mathcal{O}_{\mathbb{C}_p}))=\mathrm{M}_2(\mathbb{C}_p) \ar[r] & T_x(\mathbb{P}^1(\mathbb{C}_p)) \ar[r] & 0,}
\end{equation}
using the fact that submersions induce surjective maps on tangent spaces.
To obtain a $\mathbb{Q}_p$-rational result from this, we recall that $i\colon \Lie(\mathbb{S}_2)\subseteq \Lie(\mathrm{Gl}_2(\mathcal{O}_{\mathbb{C}_p}))$ is a $\mathbb{Q}_p$-form, i.e., the $\mathbb{C}_p$-linear extension of $i$ is an isomorphism.
Furthermore, the stabilizer of $x$ in $\mathbb{S}_2$ is $\mathbb{S}_2\cap g^{-1}\mathrm{P}(\mathcal{O}_{\mathbb{C}_p})g$, and similarly for Lie-algebras. Thus, pulling back \eqref{eq:1} along $i$ and extending 
$\mathbb{C}_p$-linearly yields a commutative diagram with exact rows
\[
\xymatrix@C-=0.4cm{ 0 \ar[r] & \Lie(g^{-1}\mathrm{P}(\mathcal{O}_{\mathbb{C}_p})g)  \ar[r] & \Lie(\mathrm{Gl}_2(\mathcal{O}_{\mathbb{C}_p}))=\mathrm{M}_2(\mathbb{C}_p) \ar[r] & T_x(\mathbb{P}^1(\mathbb{C}_p)) \ar[r] & 0\\ 
0 \ar[r] & \ar[u]^\mu  \Lie(\mathbb{S}_2\cap g^{-1}\mathrm{P}(\mathcal{O}_{\mathbb{C}_p})g)\otimes_{\mathbb{Q}_p}\mathbb{C}_p\ar[r] & \Lie(\mathbb{S}_2)\otimes_{\mathbb{Q}_p}\mathbb{C}_p
\ar[r]\ar[u]^{\simeq} & T_x(x\mathbb{S}_2)\otimes_{\mathbb{Q}_p}\mathbb{C}_p \ar[u]^\iota \ar[r] & 0.  }                     
\]
This shows that $\iota$ is surjective and implies the claim, namely that $T_x(x\mathbb{S}_2)\neq 0$.
For clarity, observe that $\iota$ will not in general be injective, but rather the snake lemma identifies the kernel of $\iota$ with the cokernel of $\mu$, the size of which measures the irrationality of $x$ (or rather its representative $g$).
\end{proof}

\begin{proof}[Alternate argument concerning \Cref{prop:orbitsize}]
Here we only prove that the orbit $x\mathbb{S}_2$ is infinite, which will suffice for the application. We leave the case $x=[1:0]$ to the reader and hence assume $x=[\mu:1]$ for some $\mu\in\mathbb{C}_p$.
Every $\gamma\in \mathbb{S}_2$ can be written uniquely as $\gamma=\alpha_0+\alpha_1\Pi$ with $\alpha_0\in W\mathbb{F}_{p^2}^*,\alpha_1\in W\mathbb{F}_{p^2}$
(and $\Pi$ a uniformizer of the division-algebra of invariant $1/2$ over $\mathbb{Q}_p$). The action of such an
element on $x$ is given by
\begin{equation}\label{eq:explicit_action} 
\gamma\cdot x=\left[\frac{\alpha_0\mu+\overline{\alpha}_1}{\overline{\alpha}_0+p\mu\alpha_1}:1\right], 
\end{equation}
assuming the displayed denominator is non-zero, see \cite[equation (25.13)]{gross_hopkins}. Hence, our claim is that the subset
\[ 
\Sigma:=\left\{ \frac{\alpha_0\mu+\overline{\alpha}_1}{\overline{\alpha}_0+p\mu\alpha_1}\,\mid\, \alpha_0\in W\mathbb{F}_{p^2}^*,\alpha_1\in W\mathbb{F}_{p^2}\right\}\subseteq\mathbb{C}_p
\]
is infinite. (Here, an overbar denotes the Galois automorphism).
Note that for $(\alpha_0,\alpha_1)=(1,0)$, the fraction takes the value $\mu$. We now write down a one-parameter
curve through $\mu$ inside $\Sigma$ by choosing $(\alpha_0,\alpha_1)=(1,z)$, namely
\[ 
f_\mu \colon \mathbb{Z}_p\longrightarrow\Sigma,\,\, f_\mu(z):=\frac{\mu+z}{1+p \mu z}
\]
(or rather the restriction to some neighborhood of $0\in\mathbb{Z}_p$ on which the denominator does not vanish,
we leave this detail to the reader).
We then compute $f_\mu'(0)=1-p\mu^2$ and leaving the case $\mu^2=-1/p$ to the reader, can assume that
$f_\mu'(0)\neq 0$. Then the inverse function theorem \cite[Part II, Chapter III, 9]{serre_lie} implies in particular that
there is an open neighborhood $0\in U=p^N\mathbb{Z}_p$ such that $f_\mu$ restricted to $U$ is injective.
Hence, $\Sigma$ is (uncountably) infinite.
\end{proof}

We now get a first consequence for $\mathbb{G}_2$-invariant ideals of $E_0$.

\begin{proposition}\label{prop:invariant_principal}
If $(0)\neq (F)\subseteq E_0\cong W\mathbb{F}_{p^2}[\![X]\!]$ is a $\mathbb{G}_2$-invariant principal ideal, then there is an $n\ge 1$ such that $(F)=(p^n)$.
\end{proposition}

\begin{proof}
By the $p$-adic Weierstrass preparation theorem \cite[Theorem 7.3]{washington}, we can (uniquely) write $F=p^nPU$ in $W\mathbb{F}_{p^2}[\![X]\!]$ with $P$ a distinguished polynomial, $U$ a unit and $n\ge 0$. Thus, the claim is that $P$ is constant.
If not, it admits a zero in $\overline{\mathbb{Q}}_p$ which, since $P$ is distinguished, must be in the maximal ideal of $\overline{\mathbb{Z}}_p$, i.e., $x$ is a point in the open unit disc $\mathfrak{X}$. Since $(F)$ is invariant, $F$ vanishes at every point of the orbit $x\mathbb{G}_2\subseteq \mathfrak{X}$. This orbit is infinite, because its image under $\Phi$ is, by \Cref{prop:orbitsize} (and recalling that $\mathbb{G}_2\supseteq\mathbb{S}_2$). Thus, the polynomial $P$ has infinitely many zeros and $P=0$, a contradiction (since $F\neq 0$).
\end{proof}

With these preparations, we can determine the invariant radical ideals in height $2$.

\begin{proof}[Proof of \Cref{thm:invariant_primes}]
We can assume that $I\neq (0),(1)$. We freely use some facts about primary decomposition of ideals
in Noetherian rings, see for example \cite[Section 6]{matsumura}.
The set of primes associated with $I$ is a finite set of primes
$\Ass(I)=\{ \mathfrak{p}=\mathfrak{p}_1,\ldots, \mathfrak{p}_n\}$. The action of $\mathbb{G}_2$ permutes $\Ass(I)$ because if
$I=\cap_{i=1}^n\mathfrak{q}_i$ is a minimal primary decomposition with $\mathfrak{p}_i=\sqrt{\mathfrak{q}_i}$, then so is
$I=g(I)=\cap_{i=1}^n g(\mathfrak{q}_i)$ for every $g\in \mathbb{G}_2$. Since all prime ideals of $E_0\cong W\mathbb{F}_{p^2}[\![X]\!]$ except $(p,X)$ are principal (as the ring is a UFD, and so all prime ideals of height 1 are principal, see \cite[Theorem 20.1]{matsumura}), 
we need to show that if
$\mathfrak{p}_1=(F)\neq (0)$ is principal, then $\mathfrak{p}_1=(p)$. The (finite) product of all $\mathbb{G}_2$-conjugates of $\mathfrak{p}_1$
is a non-zero $\mathbb{G}_2$-invariant principal ideal, say equal to $(H)\neq (0)$. Now \Cref{prop:invariant_principal} implies that $(H)=(p^n)$ for some $n\ge 1$.
Since the irreducible element $F\in E_0\cong W\mathbb{F}_{p^2}[\![X]\!]$ divides $H$, we conclude that $(F)=(p)$.
\end{proof}

\section{Descent for Balmer spectra}\label{sec:descent}

In this section we investigate descent properties of the Balmer spectrum, showing that an overly optimistic generalization of \cite[Theorem 1.3]{Balmer2016Separable} does not hold. To provide some context for our result, we observe that if $A\longrightarrow B$ is a faithfully flat map of (discrete) commutative rings, then 
the canonical diagram of Zariski spectra
\[ 
\xymatrix{\Spec(B\otimes_A B) \ar@<0.5ex>[r] \ar@<-0.5ex>[r] & \Spec(B) \ar[r] & \Spec(A)}
\]
is a coequalizer of topological spaces: This follows from $\Spec(B)\longrightarrow
\Spec(A)$ being a topological quotient map (\cite[Expose VIII, Corollaire 4.3]{SGA1} and \cite[EGA IV, Corollaire 2.3.12]{EGA4}, see also 
\cite[\href{https://stacks.math.columbia.edu/tag/02JY}{Tag 02JY}]{stacks-project}) and the more 
elementary fact that the canonical map $\Spec(B\otimes_A B)\longrightarrow\Spec(B)\times_{\Spec(A)}\Spec(B)$ is surjective.\footnote{Beware that this fiber product is formed in topological spaces, not in (affine) schemes.}
This result admits an equivalent reformulation in categorical terms: Since there is a functorial
homeomorphism $\Spec(A)\simeq\Spc(\Mod_A^\omega)$ \cite[Theorem 3.15]{thomason_classification} and an equivalence of symmetric monoidal $\infty$-categories \cite[Theorem 6.1]{lurie_dag8}
\[
\Mod_A^\omega\simeq\Tot \Mod^\omega_{B^{\otimes_A \bullet+1}},
\]
we see that applying $\Spc(-)$ to this limit diagram yields a colimit diagram
in low degrees (note that $B$ is a flat $A$-module, so the tensor products $B^{\otimes_A \bullet+1}$ are discrete commutative rings to which Thomason's result applies). The task of finding the correct generality of this categorical result was taken
up by Balmer, whose \cite[Theorem 1.3]{Balmer2016Separable} roughly implies that if $A\in\Calg(\cal{C})$ is a descendable commutative algebra object in a small symmetric monoidal $\infty$-category $\cal{C}$ (that is, the thick tensor-ideal generated by $A$ is all of $\cal{C}$), then there is a coequalizer diagram of topological spaces
\[ 
\xymatrix{\Spc(\Mod_{A\otimes A}(\cal{C})) \ar@<0.5ex>[r] \ar@<-0.5ex>[r] & \Spc(\Mod_A(\cal{C})) \ar[r] & \Spc(\cal{C}).}
\]
This covers for example the case of a finite, faithful Galois extension of $\mathbb{E}_\infty$-rings, but falls short of reproducing the classical result: There are many faithfully flat
maps $A\to B$ of (discrete) commutative rings, which do not turn $B$ into a perfect (complex
of) $A$-module(s), hence Balmer's result cannot be applied to this situation in general.
With the classical result in mind, one might hope that it is possible to weaken the finiteness assumptions in Balmer's
result, but we construct an example which seems to caution this, see
\cref{prop:nodescent}.

\subsection{Categorical descent for dualizable \texorpdfstring{$K(n)$}{K(n)}-local spectra}\label{sec:descentkn}

We remind the reader that we let $E = E_n$ be Morava $E$-theory of height $n$ at a fixed prime $p$. Viewed as a commutative ring spectrum in the $K(n)$-local category $\Sp_{K(n)}$, we denote the associated $K(n)$-local Amitsur complex by $E^{\htimes \bullet +1}\colon \Delta \to \CAlg(\Sp_{K(n)})$. Passing to categories of modules internal to $\Sp_{K(n)}$ (see \cite[Section 4.5]{lurie-higher-algebra}) thus induces a coaugmented cosimplicial diagram 
\begin{equation}\label{eq:dualizabledescent}
    \Sp_{K(n)} \to \Mod_{E^{\htimes \bullet +1}}(\Sp_{K(n)})
\end{equation}
in the $\infty$-category of compactly generated symmetric monoidal $\infty$-categories. 

\begin{proposition}\label{prop:dualizabledescent}
The coaugmention in the above diagram \eqref{eq:dualizabledescent} induces an equivalence $\Sp_{K(n)}^{\dual} \simeq \Tot\Mod_{E^{\htimes \bullet +1}}^{\omega}$ of symmetric monoidal $\infty$-categories. 
\end{proposition}
\begin{proof}
By \cref{prop:tensor_nilpotence} (see also \cite[Proposition 10.10]{mathew_galois}),  \eqref{eq:dualizabledescent} is a limit diagram. Since this diagram is constructed in the $\infty$-category of symmetric monoidal $\infty$-categories, we obtain a symmetric monoidal equivalence
\[
\xymatrix{\Sp_{K(n)}^{\dual} \ar[r]^-{\sim} & \Tot\Mod_{E^{\htimes \bullet +1}}^{\dual}(\Sp_{K(n)})}
\]
of the full subcategories of dualizable objects, using \cite[Proposition 4.6.1.11]{lurie-higher-algebra} to commute 
the formation of duals with the limit.

In order to prove the claim, consider the restriction of the right-hand side of \eqref{eq:dualizabledescent} to perfect objects. This gives rise to a natural transformation of cosimplicial diagrams
\[
\phi^{\bullet+1}\colon \Mod_{E^{\htimes \bullet +1}}^{\omega} \to \Mod_{E^{\htimes \bullet +1}}^{\dual}(\Sp_{K(n)}).
\]
The transformation $\phi^{\bullet+1}$ is fully faithful in each degree; moreover, \cref{prop:dual_perfect} shows that $\phi^0$ is also essentially surjective. It thus follows that $\Tot(\phi^{\bullet+1})$ is an equivalence, see for example \cite[Lemma 5.16]{bss_ultra1}. Consequently, there is an equivalence of symmetric monoidal $\infty$-categories
\[
\Sp_{K(n)}^{\dual} \simeq \Tot\Mod_{E^{\htimes \bullet +1}}^{\dual}(\Sp_{K(n)}) \simeq \Tot\Mod_{E^{\htimes \bullet +1}}^{\omega}
\]
as desired.
\end{proof}

\subsection{Coherence of \texorpdfstring{$E$}{E}-theory cooperations}\label{sec:cohcoop}

It follows from \cite[Theorem 2.23]{hoveystrickland_memoir} that the ring of operations $E^*E$ for Morava $E$-theory is left Noetherian. In contrast, it is known that the graded commutative ring of completed cooperations $E_*^{\vee}E$ is not Noetherian for $n>0$. Indeed, this follows from $E_*^{\vee}E/I_n \cong K_*E$ and the known description of the latter \cite[Section 2.3]{hoveystrickland_memoir}. We will, however, show in this section that $E_*^{\vee}E$ is in fact graded coherent. We start with two general results about coherent rings. 

\begin{lemma}(\cite[Theorem 2.3.3]{glaz_coherentrings})\label{lem:coherencecriterion1}
Suppose $(R_i)_{i\in I}$ is a directed diagram of commutative rings satisfying the following two conditions:
    \begin{enumerate}
        \item The rings $R_i$ are coherent for all $i \in I$.
        \item For all $i \to j$ in $I$, the ring homomorphism $R_i \to R_j$ is flat. 
    \end{enumerate}
Then $R = \colim_{I}R_i$ is coherent. 
\end{lemma}

We continue with a few preliminary remarks about filtered rings and modules as well as their associated graded; a reference is \cite[III.2.1--4]{bourbaki_commalg}. Suppose $R$ is a filtered ring and write $\Mod_R^{\filt}$ for the category of filtered discrete $R$-modules, i.e., $R$-modules with a (decreasing) filtration which is compatible with the given (decreasing) filtration on $R$. The associated graded $\gr$ defines a functor to graded $R$-modules defined by 
\[
 (\Fil_kM)_{k\ge 0}  \mapsto \bigoplus_{k\ge 0}(\Fil_kM/\Fil_{k+1}M).
\]
If $R$ is a commutative ring and $I \subseteq R$ an ideal, we can consider the $I$-adic filtration on $R$ and write $\gr_IR = \bigoplus_{k \ge 0}I^{k}/I^{k+1}$ for the associated graded ring. More generally, every $R$-module $M$ carries a filtration derived from the $I$-adic filtration on $R$, so that the $k$-th filtration step is given by $I^kM$. The associated graded $\gr_IM = \bigoplus_{k \ge 0}I^{k}M/I^{k+1}M$ is then naturally a $\gr_IR$-module.  

Consider a short exact sequence of $R$-modules
\[
0 \to M_1 \to M_2 \to M_3 \to 0.
\]
If $R$ is equipped with the $I$-adic filtration as above, this sequence can be promoted to a sequence of filtered $R$-modules as follows: Equip $M_2$ with the exhaustive filtration derived from $R$, i.e., the $I$-adic filtration. As submodule and quotient, $M_1$ and $M_3$ inherit exhaustive filtrations from $M_2$, respectively, making this sequence a short exact sequence of filtered $R$-modules. Furthermore, the filtration on $M_1$ is separated if the one on $M_2$ is.

As a special case of \cite[III.2.4, Proposition 4]{bourbaki_commalg}, we thus obtain a short exact sequence of $\gr_IR$-modules
\[
0 \to \gr M_1 \to \gr_IM_2 \to \gr_IM_3 \to 0.
\]
We emphasize, as the notation suggests, that the induced filtration on $M_3$ is the derived (i.e., $I$-adic) one, while the induced filtration in general does not coincide with the $I$-adic filtration on $M_1$. Therefore, in general, $\gr M_1$ is not isomorphic to $\gr_I M_1$.

The analogue of the following lemma for Noetherian rings was proven in \cite[III.3.9, Corollary 2]{bourbaki_commalg}.

\begin{lemma}\label{lem:coherencecriterion2}
Let $R$ be a commutative ring which is $I$-adically complete with respect to an ideal $I \subseteq R$. If $\gr_IR$ is a coherent ring, then so is $R$.
\end{lemma}
\begin{proof}
Let $J \subseteq R$ be a finitely generated ideal, and choose a short exact sequence of filtered $R$-modules
\[
0 \to K \to R^d \to J \to 0,
\]
as above. Since the $I$-adic filtration on $R^d$ is separated, so is the induced filtration on $K$. We thus obtain a short exact sequence of $\gr_IR$-modules
\[
0 \to \gr K \to (\gr_IR)^d \to \gr_IJ \to 0.
\]
Since $\gr_IR$ is coherent and $\gr_IJ$ is a finitely generated ideal in $\gr_IR$, it must be finitely presented, hence $\gr K$ is finitely generated. It follows from \cite[III.2.9, Corollary 1]{bourbaki_commalg} that $K$ is finitely generated over $R$, so $J$ is a finitely presented ideal in $R$.
\end{proof}

We now return to the case of interest. We write $A = E^{\vee}_*E$ for the graded commutative $E_*$-algebra of cooperations of $E$, and denote the degree 0 part of $A$ by $A_0$.  Note that both $A$ and $A_0$ are pro-free (i.e., the completion of a free $E$-module), see \cite[Proposition 2.2]{hovey_operations}. It follows that $A_0^{\htimes k} = A_0^{\htimes_{E_0} k}$ is pro-free for $k \ge 0$.

\begin{lemma}\label{lem:grcoherent}
For any $k\ge 0$, the graded ring $\gr_I(A_0^{\htimes k})$ is coherent.
\end{lemma}
\begin{proof}
We observe that the $E_0$-module $A_0^{\htimes k}$ is pro-free, $I$ is a regular ideal in $A_0^{\htimes k}$, and $A_0^{\htimes k}$ is $I$-adically complete, see \cite[Theorem A.9]{hoveystrickland_memoir}. Therefore, there is a natural isomorphism
\[
\gr_I(A_0^{\htimes k}) \cong \left( A_0^{\htimes k}/I\right) [x_1,\ldots,x_n],
\]
see for example \cite[Theorem 1.1.8]{brunsherzog_cmr}. Hovey proves in \cite[Proposition 3.12 and Theorem 3.13]{hovey_operations} that $A_0/I \cong K_0E$ is an ind-{\'e}tale algebra: it can be written as sequential colimit
\[
A_0/I \cong \colim_j B_j,
\]
where 
\[
B_j = \F_{p^n}[t_0,\ldots,t_j]/(t_0^{p^n-1}-1,t_1^{p^n}-t_1,\ldots,t_j^{p^n}-t_j) \htimes \F_{p^n}
\]
is an {\'e}tale algebra over $\F_{p^n}$ for all $j$; see also the proof of \cite[Theorem 12]{strickland_grosshopkins}. In particular, the transition maps $B_j \to B_{j+1}$ are flat for all $j$. Combining these two displayed isomorphisms, we thus obtain a corresponding description
\begin{align*}
    \gr_I(A_0^{\htimes k}) & \cong (A_0^{\htimes k}/I)[x_1,\ldots,x_n] \\
    & \cong (K_0E)^{\htimes_{K_0}k}[x_1,\ldots,x_n] \\
    & \cong \colim_j (B_j^{\htimes_{K_0} k}[x_1,\ldots,x_n]).
\end{align*}
This exhibits $\gr_I(A_0^{\htimes k})$ as a sequential colimit of Noetherian algebras with flat transition maps, hence it is coherent by \Cref{lem:coherencecriterion1}. Therefore, $\gr_I(A_0^{\htimes k})$ is coherent. 
\end{proof}

\begin{remark}
There are counterexamples to the analogue of the Hilbert basis theorem for coherent rings, which is why we appeal to the argument involving the ind-{\'e}tale description of $K_0E$.
\end{remark}

\begin{theorem}\label{thm:coherenthighercoop}
For any $k \ge 0$, the graded commutative ring $A^{\htimes k}$ is coherent.
\end{theorem}

\begin{proof}
The graded ring $A^{\htimes k}$ is even and 2-periodic, so it suffices to show that $A_0^{\htimes k}$ is coherent. This is an immediate consequence of \Cref{lem:grcoherent} and \Cref{lem:coherencecriterion2}.
\end{proof}

\subsection{Failure of naive descent for Balmer spectra}

We recall from \Cref{prop:dualizabledescent} that there is an equivalence $\Sp_{K(n)}^{\dual} \simeq \Tot\Mod_{E^{\htimes \bullet +1}}^{\omega}$ of symmetric monoidal $\infty$-categories and that $E \in \CAlg(\Sp_{K(n)})$ is descendable, but that $E \not \in \Sp_{K(n)}^{\dual}$, so that Balmer's result does not apply. We begin with a partial generalization of \cref{prop:etheorycomparison}, where we remind the reader that $A = E^{\vee}_*E$. 

\begin{proposition}\label{prop:ektheorycomparison}
Let $k \ge 0$. The graded comparison map 
\[
\rho^*_{\Mod_{E^{\htimes k+1}}^{\omega}}\colon\Spc(\Mod_{E^{\htimes k+1}}^{\omega}) \to \Spec^h(A^{\htimes k}) \cong \Spec(A_0^{\htimes k})
\]
is surjective. 
\end{proposition}
\begin{proof}
Because $A$ is pro-free, there is a canonical isomorphism $\pi_*E^{\htimes k+1} \cong A^{\htimes k}$. The target of Balmer's graded comparison map
\[
\Spc(\Mod_{E^{\htimes k+1}}^{\omega}) \to \Spec^h(\pi_*E^{\htimes k+1})
\]
thus identifies with $\Spec^h(A^{\htimes k})$. We have seen in \Cref{thm:coherenthighercoop} that $A^{\htimes k}$ is coherent, hence the comparison map is surjective by \Cref{thm:balmer_comparison}. The stipulated isomorphism of Zariski spectra holds because $A^{\htimes k}$ is even and 2-periodic, using \cref{lem:gradedungradedcomparison}.
\end{proof}

\begin{corollary}\label{prop:abstractcomparison}
Consider the following diagram:
\[
\xymatrix{\Spc(\Mod^{\omega}_{E \widehat{\otimes} E}) \ar@<0.5ex>[r]^{\pi_1} \ar@<-0.5ex>[r]_{\pi_2}  \ar@{->>}[d]^{\rho_2} & \Spc(\Mod^{\omega}_E) \ar[d]_{\rho_1}^{\cong} \\
\Spec(E_0^\vee E) \ar@<0.5ex>[r]^{\pi_1'} \ar@<-0.5ex>[r]_{\pi_2'} & \Spec(E_0).}
\]
Here, the $\rho_i$ are the comparison maps, and the $\pi_i$ and $\pi_i'$ are induced by the obvious base-change functors.
Then this diagram commutes for $i=1,2$, $\rho_1$ is a homeomorphism, $\rho_2$ is surjective, and the map induced on coequalizers of the two pairs of parallel arrows is a homeomorphism.
\end{corollary}
\begin{proof}
The diagram commutes by the naturality of the comparison maps. We have already seen in \cref{prop:etheorycomparison} that $\rho_1$ is a homeomorphism
because $E_*$ is even, regular and Noetherian. Next, $\rho_2$ is surjective by \cref{prop:ektheorycomparison}. The 
consequence for coequalizers is elementary and left to the reader: One might start by noting that the map on coequalizers
is an open surjection, because $\rho_1$ is a homeomorphism, and that it is injective because $\rho_2$ is surjective.
\end{proof}

\begin{remark}
We recall that the category of spectral spaces $\SpecTop \subseteq \Top$ is a reflexive subcategory, and so there exists a functor $\Top \to \SpecTop$ that is left adjoint to the inclusion, see \cite[Section 11.1]{dickmann}. We call this the spectrification functor. 
\end{remark}
Finally, we present our example. 
\begin{proposition}\label{prop:nodescent}
For height $n=2$, the diagram 
\[ 
\xymatrix{\Spc(\Mod_{E\htimes E}^\omega) \ar@<0.5ex>[r] \ar@<-0.5ex>[r] & \Spc(\Mod_E^\omega) \ar[r]^-{\pi} & \Spc(\Sp_{K(2)}^{\dual})}
\]
is not a coequalizer of topological spaces. More precisely: While $\pi$ is a topological quotient map (even an open surjection), the coequalizer of the first two terms is infinite, but $\Spc(\Sp_{K(2)}^{\dual})$ consists of three points.
However, the coequalizer computed in spectral spaces is identified with $\Spc(\Sp_{K(2)}^{\dual})$.
\end{proposition}
\begin{proof}
The coequalizer of the diagram under consideration, namely
\[ 
\xymatrix{\Spc(\Mod_{E\htimes E}^\omega)\ar@<0.5ex>[r] \ar@<-0.5ex>[r] & \Spc(\Mod_E^\omega)}
\]
identifies by \Cref{prop:abstractcomparison} with the coequalizer of 
\[ 
\xymatrix{\Spec(E_0^{\vee}E)\ar@<0.5ex>[r] \ar@<-0.5ex>[r] & \Spec(E_0).}
\]
Hovey showed this is the coaction diagram associated with the action of $\mathbb{G}_2$ on $\Spec(E_0)$
\cite[Theorem 4.11 and Section 6.3]{hovey_operations}. Hence, the coequalizer is $\Spec(E_0)/\mathbb{G}_2$, the set of orbits of
prime ideals in $E_0$ under the action of $\mathbb{G}_2$. The induced map on coequalizers takes the form
\[ 
\widetilde \pi \colon \Spec(E_0)/\mathbb{G}_2\longrightarrow\Spc(\Sp_{K(2)}^{\dual}),
\]
and \Cref{cor:surjective_balmer} shows that it is given by
\[ 
\widetilde  \pi([\mathfrak p])=\{ X\in\Sp_{K(n)}^{\dual}\,\mid\, (E_0^\vee(X)\oplus E_1^\vee(X))_{\mathfrak p}=0\}.
\]
This lets us make $\widetilde  \pi$ completely explicit, as follows: We have $\widetilde \pi([0])=\cal{D}_1=\langle L_{K(2)}S^0/p\rangle$, $\widetilde \pi([(p)])=\cal{D}_2=\langle L_{K(2)}F(2)\rangle$,
$\widetilde \pi([(p,X)])=\cal{D}_3=(0)$, and all remaining primes (which correspond to irreducible distinguished polynomials)
are mapped to $\cal{D}_1$. This shows that the continuous map $\widetilde \pi$ is surjective, and an easy inspection,
left to the reader, shows that it is open. Since the $\mathbb{G}_2$-action cannot identify the primes generated 
by two irreducible polynomials of different degrees, we see that $\Spec(E_0)/\mathbb{G}_2$ is infinite. To address the final claim, we consider the factorization of $\tilde{\pi}$
$$ \tilde{\pi}\colon\Spec(E_0)/\mathbb{G}_2\longrightarrow T\longrightarrow\Spc(\Sp_{K(2)}^{\dual})$$
through the spectrification $T$ of the topological quotient $\Spec(E_0)/\mathbb{G}_2$, claiming that the second map is a homeomorphism.
The main point here is that the orbit $[(f)]$ of any irreducible distinguished polynomial $f$ gets mapped in $T$ to the image of $[0]$: Since $T$ is in particular sober and using the description of the universal map to a sober space from \cite[\href{https://stacks.math.columbia.edu/tag/0A2N}{Tag 0A2N}]{stacks-project}, to see this we have to show that the two elements $[0]$, $[(f)]$ have the same closure in $\Spec(E_0)/\mathbb{G}_2$, i.e., that the point $[(f)]$ is dense. For this, we have to argue that the orbit $(f)\mathbb{G}_2\subseteq\Spec(E_0)$ is dense, which follows from it being infinite (see \Cref{prop:orbitsize}). We can conclude that the second map is a continuous bijection between spectral spaces and using that $T$ carries the quotient topology induced by $\Spec(E_0)$, one checks that it is a homeomorphism. 
\end{proof}

\begin{remark}
In fact, it is possible to prove that the canonical maps 
\[
\xymatrix{\Spc(\Mod_{E\otimes E}^{\omega}) \ar@<0.5ex>[r] \ar@<-0.5ex>[r] & \Spc(\Mod_E^{\omega}) \ar@{->>}[r] & \Spc(\Sp_{E}^{\omega}),}
\]
form a coequalizer diagram of topological spaces at all heights and primes, in contrast to the $K(2)$-local situation studied above.  
\end{remark}

\bibliography{balmer}\bibliographystyle{alpha}
\end{document}